\documentclass[9pt, BCOR=2mm]{amsart}
\usepackage{amssymb}
\usepackage{amscd} 
\usepackage{empheq}
\usepackage{enumitem}
\usepackage{tikz-cd}
\usepackage{multicol,caption}
\usepackage[toc,page]{appendix}
\usepackage{nomencl}
\makenomenclature
\usepackage[
  margin=3.6cm,
  includefoot,
  footskip=16pt,
]{geometry}
\usepackage{mathscinet}
\usepackage[nobysame,alphabetic,initials,msc-links,lite,backrefs]{amsrefs}
\usepackage[colorinlistoftodos,prependcaption,textsize=tiny]{todonotes}
\usepackage{mathtools, mathrsfs, color} 
\usepackage{graphicx}
\usepackage{tikz,tikz-cd}
\usepackage[all,cmtip]{xy}
\usepackage{cleveref}
\numberwithin{equation}{section}

\def\today{\number\day\space\ifcase\month\or   January\or February\or
   March\or April\or May\or June\or   July\or August\or September\or
   October\or November\or December\fi\   \number\year}

\numberwithin{equation}{section}
\newtheorem{theorem}{Theorem}[section]
\newtheorem{lemma}[theorem]{Lemma}
\newtheorem{proposition}[theorem]{Proposition}
\newtheorem{corollary}[theorem]{Corollary}
\newtheorem{definition}[theorem]{Definition}
\newtheorem{remark}[theorem]{Remark}

\newtheorem{example}[theorem]{Example}
\newcommand{\Z}{{\mathbb{Z}}}
\newcommand{\R}{{\mathbb{R}}}

\newcommand{\T}{{\mathbb{T}}}

\newcommand{\pf}{{\operatorname{pf}}}
\newcommand{\deta}{{\operatorname{det}}}
\newcommand{\K}{\mathrm{K}} 

\newcommand{\suchthat}{\middle |}

\newcommand{\M}{\mathrm{M}}

\newcommand{\Ne}{{\mathbb{Z}}_{\ge 0}}

\DeclareMathOperator{\Aut}{Aut}
\DeclareMathOperator{\Tr}{Tr}

\DeclareMathOperator{\ev}{ev}

\DeclareMathOperator{\id}{id}

\newcommand{\bn}{\noindent \begin{nummer} \rm}
\newcommand{\en}{\end{nummer}}
\newcommand{\m}[2]{\ensuremath{{#1}_{#2}}}

\title{Symmetrized non-commutative tori revisited}  

\author[S.~Chakraborty]{Sayan Chakraborty} 
\address{IAI, TCG CREST, 1st floor, Bengal Eco Intelligent Park, EM Block, Sector V, Bidhannagar, Kolkata, West Bengal 700091.}
\email{sayan.chakraborty@tcgcrest.org}

\keywords{noncommutative torus, C$^*$-crossed product, group actions, K-theory}
\subjclass[2010]{46L35, 46L55, 46L80}

\begin{document}

\begin{abstract}
For the flip action of $\Z_2$ on an $n$-dimensional noncommutative torus $A_\theta,$ using an exact sequence by Natsume, we compute the K-theory groups of $A_\theta \rtimes \Z_2$. The novelty of our method is that it also provides an explicit basis of $\K_{0}(A_{\theta}\rtimes \Z_2),$ for any $\theta.$ As an application, for a simple $n$-dimensional torus $A_{\theta},$ using classification techniques, we determine the isomorphism class of $A_{\theta}\rtimes \Z_2$ in terms of the isomorphism class of $A_{\theta}.$

\end{abstract}

\maketitle \pagestyle{myheadings} \markboth{S.~Chakraborty}{Symmetrized noncommutative tori}

\section*{Introduction}\label{sec:intro}

For $n \geq 2$, let $\mathcal{T}_{n}$ denote the space of all $n \times n$ real skew-symmetric  matrices. The $n$-dimensional noncommutative torus $A_{\theta}$ is
the universal C$^*$-algebra  generated by unitaries $U_1, U_2, U_3, \ldots, U_n$
subject to the relations 
\begin{equation}\label{eq:ccr}
	 U_jU_k = e^ {2 \pi i \theta_{jk} }U_k U_j, 
\end{equation}
for $j, k = 1, 2, 3, \ldots, n$, where $\theta:=(\theta_{jk}) \in \mathcal{T}_{n}$. For the two-dimensional noncommutative tori, since $\theta$ is determined by only one real number, $\theta_{12},$ we will often denote  the corresponding two-dimensional noncommutative torus by $A_{\theta_{12}}$. Recall that the action of $\Z_2$ on any $n$-dimensional $A_{\theta}$\,---\,often called the \emph{flip action}\,---\,is defined by sending $U_i$ to $U_i^{-1},$ for all $i.$ The study of the corresponding crossed product C$^*$-algebra $A_\theta \rtimes \Z_2$ for $n=2$ goes back to the work \cite{BEEK91}. Quickly this became one of the accessible examples of a noncommutative space. The algebra $A_\theta \rtimes \Z_2,$ for a general $n,$ also appears in M(atrix) theory and String theory, see \cite{KS00}, \cite{KS02}. The Morita equivalence classes and the isomorphism classes of $A_\theta \rtimes \Z_2$ play an important role in \cite{KS00}, \cite{KS02}. The K-theory of $A_\theta \rtimes \Z_2$ was computed by Kumjian (\cite{Kum90}) for the two-dimensional cases. Kumjian used an exact sequence of Natsume to compute such K-theory groups. Later, using the similar methods, Farsi and Watling, in \cite{FW93}, have computed the K-theory of $A_\theta \rtimes \Z_2$ for general $n$, and for a totally irrational $\theta$ (see Definition~\ref{df totally irrational}).  However, in \cite{CL19}, a major gap was pointed out in the paper \cite{FW93}. Recently this gap was rectified in \cite{CFW22} indirectly  using a result of \cite{Boc97}.  Using the tools developed by the authors in \cite{CL19}, we rectify the gap directly in this paper for general $n.$ Note that this direct method also gives us a basis of $\K_0 (A_{\theta}\rtimes  \Z_2),$ whereas the indirect method in \cite{CFW22} just computes the dimensions of  $\K_0 (A_{\theta}\rtimes  \Z_2).$ Note that $\K_1(A_{\theta}\rtimes \Z_2)$ is trivial.

The authors in \cite{FW93} used the following exact sequence of Natsume (\cite{Nat85}) to compute the K-theory of $A_\theta \rtimes \Z_2,$ for a totally irrational $\theta.$  Since $A_\theta \rtimes \Z_2$ can be written as $A_{\theta'}\rtimes_\phi (\mathbb{Z}_2 * \mathbb{Z}_2 ),$ for some antisymmetric $(n-1)\times (n-1)$ matrix $\theta'$ and some action $\phi$, in this case Natsume's exact sequence looks like 

{\small \[
\begin{CD}
\K_0 (A_{\theta'}) @>{i_{1*}-i_{2*}}>>\K_0 (A_{\theta'}\rtimes  \Z_2)\oplus \K_0 (A_{\theta'}\rtimes \Z_2)@>{j_{1*}+j_{2*}}>> \K_0 (A_{\theta'}\rtimes_\phi \Z_2*\Z_2)   \\
@A{}AA & &  @VV{e_1}V            \\
 \K_1 (A_{\theta'}\rtimes_\phi \Z_2*\Z_2) @<<{j_{1*}+j_{2*}}<  \K_1 (A_{\theta'}\rtimes \Z_2)\oplus \K_1 (A_{\theta'}\rtimes \Z_2) @<<{i_{1*}-i_{2*}}< \K_1 (A_{\theta'}),
\end{CD}
\]} 
\noindent \ignorespacesafterend where $i_1, i_2, j_1,j_2$ are natural inclusions.
In the proof of the main theorem (\cite[Theorem 7]{FW93}), it was stated that the map  $i_{1*}-i_{2*}: \mathbb{Z}^{2^{n-2}} \rightarrow \mathbb{Z}^{3.2^{n-1}},$ in the $\K_0$-level, is given by $\operatorname{Diag}(\id, 0,-\id, 0)$. (Note that $\operatorname{Diag}(\id, 0,-\id, 0)$ should be replaced by $(\id, 0,-\id, 0)^t,$ as, for example, for $n=2$ the map $i_{1*}-i_{2*}$ should be a $6 \times 1$ matrix and clearly the $\operatorname{Diag}(\id, 0,-\id, 0)$ is not.) However the reference \cite[Corollary 7.9]{Rie88}, mentioned therein, does not clearly gives the result.  This was clarified in \cite[Corollary 7.2, (see also Remark 7.5)]{CL19} for $n=3$ using the description of the Chern character map (from \cite{Wal95}), and using an explicit description of the K-theory classes of $A_{\theta}\rtimes \Z_2,$ for two-dimensional $A_{\theta}.$ In \cite[Corollary 7.2]{CL19} the authors in fact show that $i_{1*}-i_{2*}: \mathbb{Z}^{2^{n-2}} \rightarrow \mathbb{Z}^{3.2^{n-1}}$ is given by $(\id, 0,-\id, 0)^t,$ for $n=3$.

Recently in a paper with Hua (\cite{CH21}), the author of this paper has found a basis of $\K_0(A_\theta)$ consisting of projections inside $A_\theta,$ for an $n \times n$ strongly totally irrational $\theta$ (see Definition~\ref{a3}). Using this basis of $\K_0(A_\theta),$ for a strongly totally irrational $\theta,$ in this paper we prove that $i_{1*}-i_{2*}: \mathbb{Z}^{2^{n-2}} \rightarrow \mathbb{Z}^{3.2^{n-1}}$ is given by $(\id, 0,-\id, 0)^t$ for general $n.$ This computes the K-theory of $A_\theta \rtimes \Z_2,$ for all strongly totally irrational $\theta.$ We invoke Morita equivalence bi-modules for higher dimensional noncommutative tori and an explicit description of the boundary map $e_1$ of the above exact sequence to conclude that $i_{1*}-i_{2*}=(\id, 0,-\id, 0)^t.$ Our method gives an explicit basis of $\K_0(A_\theta \rtimes \Z_2)$ for all strongly totally irrational $\theta,$ in terms of projections inside $A_\theta \rtimes \Z_2.$ Let $\mathcal{P}_n$ be the set as in Equation~\ref{eq:def_of_P_n}. Then results discussed above give the following. 

\begin{theorem}(Theorem\label{intro:thm1}\ref{thm:main_K_gen})
  Let $\theta$ be a strongly irrational $n\times n$ matrix. Then
  $\K_0(A_{\theta}\rtimes \Z_2)\cong \Z^{3\cdot 2^{n-1}},$ and a generating set  of $\K_0(A_{\theta}\rtimes \Z_2)$ may be  given by 
 $\{[1], [P] ~|~ P\in \mathcal{P}_n\}.$ 
  	\end{theorem}

   For general $\theta,$  using a continuous field argument from \cite{ELPW10}, and ideas from \cite{Cha20}, we construct an explicit basis of $\K_0(A_\theta \rtimes \Z_2)$ in terms of projective modules over $A_\theta \rtimes \Z_2.$ If $\mathrm{Proj}_n$ denotes the set as defined in Theorem~\ref{thm:main_K_gen_proj}, then we have the following result. 
 
   \begin{theorem}\label{intro:thm2}(Theorem~\ref{thm:main_K_gen_proj}, Corollary~\ref{thm:main_kthoeryall}) Let $\theta \in \mathcal{T}_n.$ Then $\K_0(A_{\theta}\rtimes \Z_2)\cong \Z^{3\cdot 2^{n-1}},$ and a generating set  of $\K_0(A_{\theta}\rtimes \Z_2)$ is   given by 
 $\{[1], [\mathcal{E}] ~|~ \mathcal{E}\in \mathrm{Proj}_n\}.$ 
	 \end{theorem}

 It is known that $\theta$ is non-degenerate (see Definition~\ref{def:nondegenerate}) iff $A_\theta$ is simple. When $\theta \in \mathcal{T}_{n}$ is non-degenerate,  $A_\theta\rtimes \Z_2$ is an $A F$ algebra (see \cite[Theorem 6.6]{ELPW10}, and also Corollary~\ref{cor:flip_AF}). For non-degenerate $\theta_1, \theta_2$, our construction of explicit bases of $\K_0\left(A_{\theta_i} \rtimes \mathbb{Z}_2\right), i=1,2$, allows us to construct explicit isomorphisms between the Elliott invariants of $A_{\theta_1} \rtimes \mathbb{Z}_2$ and $A_{\theta_2} \rtimes \mathbb{Z}_2$ out of an isomorphism between the Elliott invariants of $A_{\theta_1}$ and $A_{\theta_2}$, and the converse holds if, in addition, one of the $\theta_i$ is totally irrational. This results in the following theorem.

\begin{theorem}
	\label{intro:thm3}(Theorem~\ref{thm:main_iso})
		Let $\theta_{1}, \theta_{2} \in \mathcal{T}_n $ be non-degenerate.
		Let $\Z_2$ act on $A_{\theta_{1}}$ and $A_{\theta_{2}}$ by the flip actions.  Then $A_{\theta_{1}}\rtimes \Z_2$ is isomorphic to $A_{\theta_{2}}\rtimes \Z_2$ if $A_{\theta_{1}}$ is isomorphic to $A_{\theta_{2}}.$ Moreover, if any one of $\theta_{1}, \theta_{2}$ is totally irrational, the converse is true.
 	 \end{theorem}
 	 The above theorem is a generalisation of \cite[Theorem 6.4]{ELPW10} for general $n.$ It is worth mentioning that the only canonical action (in the sense of \cite{JL15}) of a finite cyclic group on a 3-dimensional torus $A_\theta,$ when $\theta$ is non-degenerate, is the flip action (\cite[Theorem 1.4]{JL15}).

It should be noted that, using a completely different approach, Davis and L\"uck (\cite{DL13}) computed the K-theory of $A_\theta \rtimes \Z_2$ when $\theta$ is the $n \times n$ zero matrix. However, from their methods it is not clear  how to extract a concrete basis for $\K_0(A_\theta \rtimes \Z_2),$ and hence a classification type result like Theorem~\ref{intro:thm2}.

This article is organised as follows. In Section~\ref{sec:definition}, we define  $A_\theta \rtimes \Z_2$ through twisted group C$^*$-algebras and study some basic properties of the crossed product. The K-theory of $A_\theta$ and a generating set of $\K_0(A_\theta),$ for a strongly totally irrational $\theta,$ are described in Section~\ref{sec:rie}. Section~\ref{sec:k-theory} deals with descriptions of the maps that appear in Natsume's exact sequence, and the proof of  Theorem~\ref{intro:thm1}. In Section~\ref{sec:gen_general} we use the continuous field approach of \cite{ELPW10} to describe the explicit generators of $\K_0(A_\theta\rtimes\Z_2)$ for general $\theta,$ and prove Theorem~\ref{intro:thm2}. The classification-type theorem,  Theorem~\ref{intro:thm3}, is proved in Section~\ref{sec:iso_classes}. In Appendix~\ref{sec:riep}, we revisit the construction of the two-dimensional Rieffel projection which is used in the main construction of Section~\ref{sec:rie} and Section~\ref{sec:k-theory}, and in Appendix~\ref{sec:strongly_irr} we give a class of examples of strongly totally irrational matrices. Finally in Appendix~\ref{sec:path}, we explicitly describe the continuous field which is used in Section~\ref{sec:gen_general}.

	\textbf {Notation}: $e(x)$ will always denote the number $e^{2\pi i x}$, and $\id_m$ (or without the `m' decoration if the context is clear) will be the $m\times m$ unit matrix.

	
	\section{$A_\theta\rtimes \Z_2$ - revisited}\label{sec:definition}
	
	Let $G$ be a discrete group.  A map $\omega: G\times G \to \mathbb{T}$ is called a \emph{2-cocycle} if 
$$ \omega(x, y) \omega(xy, z)= \omega(x, yz) \omega(y, z) \quad \text{and}  \quad \omega(x, 1) = 1 = \omega(1, x)\label{units}
\label{cocycle}
$$
for all  $x, y, z \in G .$

The \emph{$\omega$-twisted left regular representation} of the group $G$ is given by the formula:

$$(L_{\omega}(x)f)(y) = \omega(x, x^{-1}y)f(x^{-1}y),$$ for $f \in l^2(G)$. The \emph{reduced twisted group C$^*$-algebra} $C^*(G,\omega)$  is defined as the sub-C$^*$-algebra of $B(l^2(G))$ generated by the $\omega$-twisted left regular representation of the group $G$. Since we do not talk about full group C$^*$-algebras in this paper, we simply call $C^*(G,\omega)$ the twisted group C$^*$-algebra of $G$ with respect to $\omega.$ When $\omega=1,$ $C^*(G,\omega)=:C^*(G)$ is the usual reduced group C$^*$-algebra of $G.$  We refer to \cite[Section 1]{ELPW10} for more on twisted group C$^*$-algebras and the details of the above construction. 

\begin{example}\label{ex:nt}
	\normalfont
Let $G$ be the group $\Z^n$. For each $\theta \in \mathcal{T}_{n}$, construct a 2-cocycle $\omega_\theta$ on $G$ by defining $\omega_\theta(x, y) = e(\langle -\theta x,y \rangle/2 )$. The corresponding twisted group C$^*$-algebra $C^*(G, \omega_\theta)$ is isomorphic to the $n$-dimensional noncommutative torus $A_\theta$, which was defined in the introduction. The isomorphism sends $ \delta_{x_i} \in C^*(\Z^n, \omega_\theta)$ to $U_i$ where $x_i = (0, \ldots, 1, \ldots, 0),$ with 1 at the $i^{th}$ position.   

  \end{example}
\begin{example}\label{ex:orbifold}
\normalfont
 Let $\Z_2$ act on $\Z^n$ by sending $x$ to $-x$. Let us also take a $\theta \in \mathcal{T}_{n}$.  Then we can define a 2-cocycle $\omega_\theta '$ on $G:=\Z^n \rtimes \Z_2$ by $\omega_\theta '((x,s),(y,t)) = \omega_\theta (x,s\cdot y)$. By Lemma 2.1 of \cite{ELPW10} we have $C^*(\Z^n\rtimes \Z_2, \omega_\theta ') = A_{\theta} \rtimes_\beta \Z_2,$ where the action $\beta$ of $\Z_2$ on $A_\theta$ is given by sending $U_i$ to $U_i^{-1}$ which is the \emph{flip action}. For the crossed product with the flip action $A_{\theta} \rtimes_\beta \Z_2,$ we shall often drop the decoration $\beta$ from $A_{\theta} \rtimes_\beta \Z_2,$ and denote it by $A_{\theta} \rtimes \Z_2.$   
 \end{example}
 
 Let  $\theta$ be as before and let $\theta'$ be  the  upper left $(n-1) \times (n-1)$ block of $\theta$. In this case, $A_{\theta}$ can be written as a crossed product $ A_{\theta'}\rtimes_{\gamma} \Z$, where the action $\gamma$  of $\Z$ on $A_{\theta'}$ is determined on the positive generator of $\Z$ by mapping
$ U_i$ to $e(-\theta_{in})U_i,$ for $i =1, \ldots, n-1.$  Now $A_{\theta}\rtimes \Z_2 = A_{\theta'}\rtimes_\phi (\Z \rtimes \Z_2) = A_{\theta'}\rtimes_\phi \Z_2 *  \Z_2$, since $\Z_2 *  \Z_2$ is isomorphic to $\Z \rtimes \Z_2$ as groups (cf. \cite[Proposition 6]{FW93}). Note that one copy of $\Z_2$ acts on $A_{\theta'}$ by the flip action $\beta,$ and the other by $\alpha = \gamma \circ \beta$.

 \begin{lemma}\label{lemma:flip=flip}
 	$A_{\theta'}\rtimes \Z_2 \cong A_{\theta'}\rtimes_\alpha \Z_2.$
 \end{lemma}
 \begin{proof}
 	$A_{\theta'}\rtimes_\alpha \Z_2$ is generated by the unitaries $U_1, U_2, \ldots, U_{n-1}$ and $W'=U_nW$ and  we have the relations $W'^2 = 1,  U_jU_k = e(\theta_{jk})  U_kU_j, W'U_iW' = e(\theta_{in})U_i^{-1}.$ Upon setting $\tilde{U_i} = e(-\frac{1}{2}\theta_{in})U_i$ for $i=1,2,\ldots,n-1$, we have that $$\tilde{U}_j\tilde{U}_k = e(\theta_{jk})\tilde{U}_k\tilde{U}_j,\quad  W'\tilde{U_i}W' = \tilde{U_i}^{-1}.$$ So $A_{\theta'}\rtimes_\alpha \Z_2$ is isomorphic to $A_{\theta'}\rtimes \Z_2.$
 \end{proof}

 \section{The construction of Rieffel-type projections and $\K_0(A_\theta)$}\label{sec:rie}
 
For a formal definition of the pfaffian $\pf(A)$ of an $n \times n$ skew-symmetric matrix $A$ with $n=2m$ even, we refer to \cite[Definition 3.1]{Cha20}. If $n=2m$ for some integer $m\geq1,$ then for
$$A=\left(\begin{matrix}
                  0 & \theta_{12}&\cdots&&\cdots&\theta_{1n} \\
                  -\theta_{12} &\ddots&\ddots&&& \theta_{2n}\\
                  \vdots&\ddots&&&&\\
                  &&&&&\\
                  &&&&\ddots&\vdots\\
                  -\theta_{1(n-1)}&&&\ddots&\ddots& \theta_{(n-1)n}\\
                  -\theta_{1n} &\cdots&&\cdots& -\theta_{(n-1)n}&0
                \end{matrix}
\right),$$
the pfaffian of $A$ is given by $\sum_{\xi}(-1)^{|\xi|}\Pi_{s=1}^m \theta_{\xi(2s-1)\xi(2s)},$ where the sum is taken over all elements $\xi$ of the
permutation group $S_n$ such that
$\xi(2s-1)<\xi(2s)$ for all $1\leq s\leq m$ and $\xi(1)<\xi(3)<\cdots<\xi(2m-1).$

\begin{definition}\label{2p minor}Let $n\geq 2$ be an integer, and let $p$ be an integer such that $1\leq p\leq\frac{n}{2}.$ A {\it $2p$-pfaffian minor} (or just {\it pfaffian minor}) of a skew-symmetric $n\times n$ matrix $A$ is the pfaffian of a sub-matrix $\m{A}{I}$ of $A$ consisting of rows and columns indexed by $i_1,i_2,\dots,i_{2p}$ for some numbers
  $1\leq i_1<i_2<\cdots<i_{2p}\leq n,$ and $I=(i_1,i_2,\dots,i_{2p}).$ Define the length of $I,$ $|I|:=2p.$ We often use ${\pf}_{I}^A$
  as the abbreviation of ${\pf}(\m{A}{I})$ without special emphasis.
  
  For $p=0,$ define $I$ to be the empty sequence $\emptyset,$ and in this case, define  ${\pf}(\m{A}{I})={\pf}(\m{A}{\emptyset}):=1.$ The length of $I=\emptyset$ is defined to be zero. 
  
  The set of all such $I$'s, for a fixed $n$ and varying $p,~ 0\leq p\leq\frac{n}{2},$ is denoted by $\mathrm{Minor}(n).$ Of course, $\mathrm{Minor}(n)\subset \mathrm{Minor}(n+1),$ for all $n.$ Note that the number of elements of $\mathrm{Minor}(n)$ is $2^{n-1}.$
\end{definition}

Let $\Tr$ denote the canonical tracial state on $A_\theta$ satisfying $\Tr(1)=1,$  $\Tr(U_1^{m_1}U_2^{m_2}\cdots U_n^{m_n})=0$ unless $(m_1,m_2,\ldots,m_n)=0\in \Z^n.$ We recall the following fact due to Elliott which will play a key role.
\begin{theorem}[Elliott]\label{elliott_image_of_trace}

Let $\theta$ be a skew-symmetric real $n\times n$ matrix. Then there is an isomorphism $h: \K_0\left(A_{\theta}\right) \rightarrow \Lambda^{\text {even }} \mathbb{Z}^n$ such that $\exp _{\wedge}(\theta) \circ h=\Tr$, where $\operatorname{exp}_{\wedge}(\theta)$ is the exterior exponential map
 $$\operatorname{exp}_{\wedge}(\theta):\Lambda^{\operatorname{even}}\Z^n\rightarrow \R,$$and such that $h([1])$ is the standard generator $1 \in \Lambda^0\left(\mathbb{Z}^n\right)=\mathbb{Z}$. In particular, $\Tr(\K_0(A_\theta))$  is the range of the exterior exponential.
\end{theorem}
 \noindent We refer to (\cite[(1.3, Theorem 2.2, Theorem 3.1)]{Ell84}) for the definition of the exterior exponential and the proof of the above theorem. The range of the exterior exponential is well known and is given below as a corollary of the above theorem:

\begin{corollary} For $\theta\in \mathcal{T}_n,$
$\Tr(\K_0(A_{\theta}))$ is the subgroup of $\mathbb{R}$ generated by $1$ and the numbers
\begin{eqnarray}
\sum_{\xi}(-1)^{|\xi|}\Pi_{s=1}^{p}\theta_{j_{\xi (2s-1)}j_{\xi(2s)}}\nonumber
\end{eqnarray} for
$1\leq j_1<j_2<\cdots<j_{2p}\leq n,$ where the sum is taken over all elements
$\xi$ of the permutation group $S_{2p}$ such that $\xi(2s-1)<\xi(2s)$ for all
$1\leq s\leq p$ and $\xi(1)<\xi(3)<\cdots<\xi(2p-1).$
\end{corollary}

Noting that $\sum_{\xi}(-1)^{|\xi|}\Pi_{s=1}^{p}\theta_{j_{\xi (2s-1)}j_{\xi(2s)}}$ is exactly the pfaffian of $\m{\theta}{I}$,
where $I=(i_1,i_2,\dots,i_{2p})$, we have
\begin{eqnarray}\label{range of trace}
  \Tr(\K_0(A_{\theta}))=\sum\limits_{I\in \mathrm{Minor}(n)}{\pf}(\m{\theta}{I})\mathbb{Z}.
\end{eqnarray}

Except for the case $I=\emptyset,$ it is not clear whether the numbers ${\pf}(\m{\theta}{I})$ can be realized as traces of projections (which we call \emph{Rieffel-type projections}, if exist) inside $A_{\theta}$ or not. For the case $I=\emptyset,$ of course we can take the projection $1\in A_{\theta}.$  In this section we will construct Rieffel-type projections for a large class of $A_{\theta}.$


\begin{definition}[Definition 1 of \cite{FW93}]\label{df totally irrational}
  We say that $\theta\in\mathcal{T}_n$ is {\it totally irrational} if $\operatorname{exp}_{\wedge}(\theta)$ is an injective map from $\Lambda^{even}\mathbb{Z}^n$ to $\mathbb{R}$ (cf. Section 6 and 7 of \cite{Rie88}).
\end{definition}

It is clear from Elliott's work (Theorem~\ref{elliott_image_of_trace}) that $\Tr$ is injective if and only if $\theta$ is totally irrational.  Now the range of $\operatorname{exp}_{\wedge}(\theta)$ is given by
$$\sum\limits_{I\in \mathrm{Minor}(n)}{\pf}(\m{\theta}{I})\mathbb{Z}.$$ 
Thus $\theta$ is totally irrational if and only if the numbers $\mathrm{pf}(\m{\theta}{I}),$ $I\in \mathrm{Minor}(n),$ are rationally independent. Note that if $\theta$ is totally irrational, $\theta$ is also nondegenerate in the sense of \cite{Phi06} (see Definition~\ref{def:nondegenerate}), and $A_{\theta}$ is a simple $C^*$-algebra by Theorem 1.9 of \cite{Phi06}.

For $$\theta=\left(\begin{matrix}
                 \theta_{1,1} &  \theta_{1,2} \\
               \theta_{2,1}& \theta_{2,2}
                 \end{matrix}\right)= \left(\begin{matrix}
               \theta_{1,1} & \theta_{1,2} \\
                -\theta_{1,2}^t &\theta_{2,2}
                 \end{matrix}\right)\in\mathcal{T}_n,\quad n=2l ~\text{for}~l>1, $$
                such that
                \begin{eqnarray}
  \theta_{1,1}=\left(\begin{matrix}
              0 &  \theta_{12} \\
               \theta_{21}& 0
                 \end{matrix}\right)=\left(\begin{matrix}
                 0 &  \theta_{12} \\
              - \theta_{12}& 0
                 \end{matrix}\right)\in\mathcal{T}_2\nonumber
\end{eqnarray} is an invertible $2\times2$ matrix,
                 $$\theta_{2,2}=\left(\begin{matrix}
                 0 &  \theta_{34}&\cdots&\theta_{3n} \\
              - \theta_{34}& 0&\cdots&\theta_{4n}\\
              \vdots&\vdots&\ddots&\vdots\\
              -\theta_{3n}&-\theta_{4n}&\cdots&0
                 \end{matrix}\right)\in\mathcal{T}_{n-2},\quad \theta_{1,2}=\left(\begin{matrix}
                \theta_{13} &  \theta_{14}&\cdots&\theta_{1n} \\
               \theta_{23}& \theta_{24}&\cdots&\theta_{2n}
                 \end{matrix}\right),$$
     and $$\theta_{2,1}=\left(\begin{matrix}
                -\theta_{13} & - \theta_{23} \\
                -\theta_{14} & - \theta_{24} \\
                \vdots&\vdots\\
             -\theta_{1n}& -\theta_{2n}
                 \end{matrix}\right);$$
we have  $ \theta_{1,1}^{-1}=\left(\begin{matrix}
                 0 &  -\frac{1}{\theta_{12}} \\
              \frac{1}{\theta_{12}}& 0
                 \end{matrix}\right)\in\mathcal{T}_2,$ and
$$\theta_{2,2}-\theta_{2,1}\theta_{1,1}^{-1}\theta_{1,2}=\left(\begin{matrix}
                 0 &              \theta_{34}-\frac{-\theta_{23}\theta_{14}+\theta_{13}\theta_{24}}{\theta_{12}}&\cdots&\theta_{3n}-\frac{-\theta_{23}\theta_{1n}+\theta_{13}\theta_{2n}}{\theta_{12}} \\
              - \theta_{34}+\frac{-\theta_{23}\theta_{14}+\theta_{13}\theta_{24}}{\theta_{12}}& 0&\cdots&\theta_{4n}-\frac{-\theta_{24}\theta_{1n}+\theta_{14}\theta_{2n}}{\theta_{12}}\\
              \vdots&\vdots&\ddots&\vdots\\
              -\theta_{3n}+\frac{-\theta_{23}\theta_{1n}+\theta_{13}\theta_{2n}}{\theta_{12}}&-\theta_{4n}+\frac{-\theta_{24}\theta_{1n}+\theta_{14}\theta_{2n}}{\theta_{12}}&\cdots&0
                 \end{matrix}\right).$$
                 Hence we have
                \begin{eqnarray}\label{F(theta)}
             \theta_{2,2}-\theta_{2,1}\theta_{1,1}^{-1}\theta_{1,2}=\left(\begin{matrix}
                0 &  \frac{{\pf}^{\theta}_{(1,2,3,4)}}{\theta_{12}}&\cdots&\frac{{\pf}^{\theta}_{(1,2,3,n)}}{\theta_{12}} \\
              - \frac{{\pf}^{\theta}_{(1,2,3,4)}}{\theta_{12}}& 0&\cdots&\frac{{\pf}^{\theta}_{(1,2,4,n)}}{\theta_{12}}\\
              \vdots&\vdots&\ddots&\vdots\\
              -\frac{{\pf}^{\theta}_{(1,2,3,n)}}{\theta_{12}}&-\frac{{\pf}^{\theta}_{(1,2,4,n)}}{\theta_{12}}&\cdots&0
                 \end{matrix}\right).
                 \end{eqnarray}

We have the following lemma.
\begin{lemma}(\cite[Lemma 3.6]{CH21})\label{submatrix,x}
For any integer $n\geq 2,$
let
  $$\theta=\left(\begin{matrix}
                  \theta_{1,1} &  \theta_{1,2} \\
                 \theta_{2,1}& \theta_{2,2}
                 \end{matrix}
\right)=\left(\begin{matrix}
                  \theta_{1,1} &  \theta_{1,2} \\
                 -\theta_{1,2}^t& \theta_{2,2}
                 \end{matrix}
\right)\in \mathcal{T}_n,$$
where $\theta_{1,1}$ is invertible $2\times 2$ matrix, one has
\begin{eqnarray}\label{=2}
{\pf}(\theta_{1,1}){\pf}(\m{(\theta_{2,2}-\theta_{2,1}\theta_{1,1}^{-1}\theta_{1,2})}{I'})
={\pf}(\m{\theta}{I}),
\end{eqnarray}
where $I'\in \mathrm{Minor}(n-2)\setminus \{\emptyset\}$ and $I=(1,2,i_1+2,i_2+2,\dots,i_{2l}+2)$ for $I'=(i_1,i_2,\dots,i_{2l})$.
In particular, when $n$ is an even number, taking $I'=(1,2,\dots,n-2),$  we have
\begin{eqnarray}\label{pf1}
  {\pf}(\theta)
  ={\pf}(\theta_{1,1}){\pf}(\theta_{2,2}-\theta_{2,1}\theta_{1,1}^{-1}\theta_{1,2}).
\end{eqnarray}
\end{lemma}
\begin{proof}
See the proof of Lemma 3.6 in \cite{CH21}.
\end{proof}

In order to explain our symbols, we give the following definition:
\begin{definition}
  For $$\theta=\left(\begin{matrix}
                 \theta_{1,1} &  \theta_{1,2} \\
               \theta_{2,1}& \theta_{2,2}
                 \end{matrix}\right)\in\mathcal{T}_n,$$
                 where $\theta_{11}$ is an invertible $2\times2$ matrix,
                 and $n=2l,$ $l>1,$ define
                 \begin{eqnarray}\label{de F}
                 F(\theta)=
                 \theta_{2,2}-\theta_{2,1}\theta_{1,1}^{-1}\theta_{1,2}\in \mathcal{T}_{n-2}.
                 \end{eqnarray}
\end{definition}
Hence from (\ref{pf1}), we have
\begin{eqnarray}\label{q1}
{\pf}(\theta)={\pf}(\theta_{1,1}){\pf}(F(\theta)).
\end{eqnarray}

If $m$ is an integer less than $l$, we denote by $F^m$ the composition of $F$ taken $m$ times (when this makes sense) and $F^0(\theta):=\theta.$
Note that $F$ is defined for a $\theta$ such that $\theta_{11}$ is invertible, but still $F^m$ may not make sense. The following lemma tells us  when $F^m(\theta)$ is well-defined.
\begin{lemma}(\cite[Lemma 3.8, Lemma 3.10]{CH21})\label{h1}
   Let $\theta\in \mathcal{T}_n$ with $n=2l$ for $l>1.$ If ${\pf}^{\theta}_{(1,2,\dots,2s)}\neq 0$ for all $s=1,2,\dots,l-1,$ then
   $F^m(\theta)$ is well-defined for $m=1,2,\dots,l-1,$ and is given by
 
  \begin{eqnarray}\label{F^ii}
  F^m(\theta)=\left(
  \begin{array}{cccc}
  0&\frac{{\pf}^{\theta}_{(1,2,\dots,n-p-1,n-p)}}{{\pf}^{\theta}_{(1,2,\dots,n-p-2)}}&\cdots&\frac{{\pf}^{\theta}_{(1,2,\dots,n-p-1,n)}}{{\pf}^{\theta}_{(1,2,\dots,n-p-2)}}\\
  -\frac{{\pf}^{\theta}_{(1,2,\dots,n-p-1,n-p)}}{{\pf}^{\theta}_{(1,2,\dots,n-p-2)}}&0&\cdots& \frac{{\pf}^{\theta}_{(1,2,\dots,n-p,n)}}{{\pf}^{\theta}_{(1,2,\dots,n-p-2)}}\\
  \vdots&\vdots&\ddots&\vdots\\
  -\frac{{\pf}^{\theta}_{(1,2,\dots,n-p-1,n)}}{{\pf}^{\theta}_{(1,2,\dots,n-p-2)}}&
  -\frac{{\pf}^{\theta}_{(1,2,\dots,n-p,n)}}{{\pf}^{\theta}_{(1,2,\dots,n-p-2)}}&
  \ldots&0
  \end{array}\right).
  \end{eqnarray}
  In particular,
  \begin{eqnarray}\label{g1}
  F^m(\theta)_{jk}=\frac{{\pf}^{\theta}_{(1,2,\dots,s',s'+j,s'+k)}}{{\pf}^{\theta}_{(1,2,\dots,s')}},\quad p=n-2m-2,\quad s'=n-p-2=2m.
  \end{eqnarray}
  
\end{lemma}
\begin{proof}
The lemma is exactly the content of Lemma 3.8 and Lemma 3.10 of \cite{CH21}. See the proof of those.
\end{proof}

Let us use the above lemma to say more about $A_{F^{m}(\theta)}$ when $\theta\in\mathcal{T}_n$
is totally irrational with $n=2l\geq2$. When
$\theta$ is totally irrational,  the entries (above the diagonal) of $F^m(\theta)$ are all irrational and independent over $\mathbb{Q}$ for $m=0,1,\dots,l-1.$  
This is because we have ${\pf}^{\theta}_{(1,2,\dots,2s)}\neq 0$ for $s=1,2,\dots,l-1,$ by total irrationality of $\theta$. By the above lemma, we have that $ F^m(\theta)$ is well-defined and

  \begin{eqnarray}\label{g1}
  F^m(\theta)_{jk}=\frac{{\pf}^{\theta}_{(1,2,\dots,s',s'+j,s'+k)}}{{\pf}^{\theta}_{(1,2,\dots,s')}},\quad p=n-2m-2,\quad s'=n-p-2=2m,
  \end{eqnarray}
  for $m=1,\dots,l-1.$
Now since $\theta$ is totally irrational, the numbers ${\pf}^{\theta}_{(1,2,\dots,n-p-2,j,k)}$, $n-p-1\leq j<k\leq n$  along with ${\pf}^{\theta}_{(1,2,\dots,n-p-2)},$ are irrational and independent over $\mathbb{Q}$. This means that the numbers $\frac{{\pf}^{\theta}_{(1,2,\dots,n-p-2,j,k)}}{{\pf}^{\theta}_{(1,2,\dots,n-p-2)}}$ are irrational numbers. Next, to show that these numbers are independent over $\mathbb{Q}$, let us write
$$\sum\limits_{n-p-1\leq j<k\leq n}c_{j,k} \frac{{\pf}^{\theta}_{(1,2,\dots,n-p-2,j,k)}}{{\pf}^{\theta}_{(1,2,\dots,n-p-2)}}=0,\quad c_{j,k}\in \mathbb{Q}.$$
This implies
$$\sum\limits_{n-p-1\leq j<k\leq n}c_{j,k}{\pf}^{\theta}_{(1,2,\dots,n-p-2,j,k)}=0,\quad c_{j,k}\in \mathbb{Q}.$$
But since the numbers ${\pf}^{\theta}_{(1,2,\dots,n-p-2,j,k)},$ $n-p-1\leq j<k\leq n$ are independent over $\mathbb{Q},$ $c_{j,k}$'s are all zero. Hence the numbers $\frac{{\pf}^{\theta}_{(1,2,\dots,n-p-2,j,k)}}{{\pf}^{\theta}_{(1,2,\dots,n-p-2)}}$, $n-p-1\leq j<k\leq n$ are rationally
independent. So we have shown that the entries (above the diagonal) of $F^m(\theta)$ are
irrational and rationally independent. Now, using \cite[Lemma 1.7]{Phi06}, it is easy to see that
$F^m(\theta)$  is non-degenerate. Hence $A_{F^m(\theta)}$ is simple and has a unique tracial state for $m=0,1,\dots,l-1$.

In the following we shall construct Rieffel-type projections for the
higher dimensional noncommutative tori.

 In \cite{RS99}, Rieffel and Schwarz defined (densely) an action of the group $\mathrm{SO}(n, n|\Z)$ on $\mathcal{T}_{n}$. Recall that $\mathrm{SO}(n, n|\Z)$ is the subgroup of $\mathrm{SL}(2n, \Z),$ which consists all matrices $g$ with the following block form: $$
g=\left(\begin{array}{ll}
A & B \\
C & D
\end{array}\right),
$$
where $A, B, C$ and $D$ are arbitrary $n \times n$ matrices over $\Z$ satisfying
$$
A^{t} C+C^{t} A=0, \quad B^{t} D+D^{t} B=0 \quad \text { and } \quad A^{t} D+C^{t} B=\id_n.
$$
The action of $\mathrm{SO}(n, n|\Z)$ on $\mathcal{T}_{n}$ is defined as
$$
g \theta:=(A \theta+B)(C \theta+D)^{-1}
$$
whenever $C \theta+D$ is invertible. The subset of $\mathcal{T}_{n}$ on which the action of every $g \in \mathrm{SO}(n, n|\Z)$ is defined, is dense in $\mathcal{T}_{n}$ (see \cite[page 291]{RS99}). We have the following theorem due to Hanfeng Li. 
\begin{theorem}\label{thm:li_Morita}(\cite[Theorem 1.1]{Li04})
	  For any $\theta \in \mathcal{T}_{n}$ and $g \in \mathrm{SO}(n, n|\Z),$ if $g \theta$ is defined then $A_{\theta}$ and $A_{g \theta}$ are strongly Morita equivalent.
\end{theorem}

For any $R \in \mathrm{GL}(n,\Z),$ let us denote by $\rho(R)$ the matrix $\left(\begin{array}{cc}R & 0 \\ 0 & \left(R^{-1}\right)^{\mathrm{t}}\end{array}\right) \in \mathrm{SO}(n, n|\Z) ,$ and for any $N \in \mathcal{T}_{n} \cap \mathrm{M}_{n}(\Z),$ we denote by $\mu(N)$ the matrix $\left(\begin{array}{cc}\id_n & N \\ 0 & \id_n\end{array}\right) \in \mathrm{SO}(n, n|\Z).$
Notice that the noncommutative tori corresponding to the matrices $\rho(R) \theta=R \theta R^{t}$ and $\mu(N) \theta=\theta+N$ are both isomorphic to $A_{\theta}$. Also define

$$\mathrm{SO}(n, n|\Z) \ni \sigma_{2 p}:=\left(\begin{array}{cccc}0 & 0 & \mathrm{id}_{2 p} & 0 \\ 0 & \mathrm{id}_{n-2 p} & 0 & 0 \\ \mathrm{id}_{2 p} & 0 & 0 & 0 \\ 0 & 0 & 0 & \mathrm{id}_{n-2 p}\end{array}\right), 1 \leqslant p \leqslant n / 2.$$
 
 We recall the approach of Rieffel \cite{Rie88} to find the $A_{\sigma_{2 p}\theta}-A_\theta$ bimodule and follow the presentation in \cite{Li04}. Let us fix a number $p$ with $1\leq p\leq n\slash 2$ and let $q\geq0$ be an integer such that $n=2p+q.$ Let us write $\theta\in \mathcal{T}_n$ as
$$\left(\begin{matrix}
                  \theta_{1,1} & \theta_{1,2} \\
                  \theta_{2,1} & \theta_{2,2}
                \end{matrix}
\right),$$
partitioned into four sub-matrices $\theta_{1,1},\theta_{1,2},\theta_{2,1},\theta_{2,2},$ and assume $\theta_{1,1}$ to be an invertible $2p\times 2p$ matrix. Write $\sigma_{2 p}$ as $\sigma$. Then
\begin{equation}\label{eq:defoftheta'}
	\sigma(\theta)=\left(\begin{matrix}
                  \theta_{1,1}^{-1} & -\theta_{1,1}^{-1}\theta_{1,2} \\
                  \theta_{2,1}\theta_{1,1}^{-1} & \theta_{2,2}-\theta_{2,1}\theta_{1,1}^{-1}\theta_{1,2}
                \end{matrix}
\right). \end{equation} 

Set $A=A_{\theta}$ and $B=A_{\sigma(\theta)}.$
Let $\mathcal{M}$ be the group $\mathbb{R}^{p}\times \mathbb{Z}^q,$ $G=\mathcal{M}\times\widehat{\mathcal{M}}$ and
$\langle\cdot,\cdot\rangle$ be the natural pairing between $\mathcal{M}$ and its dual group $\widehat{\mathcal{M}}$ (our notation does not distinguish between the pairing of a group and its dual group, and the standard inner product on a linear space). Also, denote the linear dual of $\mathbb{R}^{k}$ by $(\mathbb{R}^k)^*.$ Consider the Schwartz space $\mathcal{E}^{\infty}=\mathscr{S}(\mathcal{M})$ consisting of smooth and rapidly decreasing complex-valued functions on $\mathcal{M}.$

Denote by $\mathcal{A}^{\infty}=A_{\theta}^{\infty}$ and $\mathcal{B}^{\infty}=A_{\sigma(\theta)}^{\infty}$ the dense sub-algebras of $A$ and $B$, respectively, consisting of formal series with rapidly decaying coefficients. Let us consider the following $(2p+2q)\times(2p+q)$
real valued matrix:
\begin{eqnarray}T=\left(\begin{matrix}
                  T_{11} & 0 \\
                 0 & {\rm id}_q\\
                 T_{31}&T_{32}
                \end{matrix}
\right),\nonumber
\end{eqnarray}
where $T_{11}$ is an invertible matrix such that $T^{t}_{11}J_0 T_{11}=\theta_{1,1},$ $J_0=\left(\begin{matrix}
                  0 & {\rm id}_{p} \\
                 -{\rm id}_{p}& 0
                 \end{matrix}
\right),$
$T_{31}=\theta_{2,1}$ and $T_{32}$ is any $q\times q$ matrix such that
$\theta_{2,2}=T_{32}-T_{32}^t.$ For our purposes, we take $T_{32}=\theta_{2,2}\slash 2.$

We also define the following $(2p+2q)\times (2p+q)$ real valued matrix:
$$S=\left(\begin{matrix}
                  J_0(T_{11}^t)^{-1} &  -J_0(T_{11}^t)^{-1} T_{31}^t \\
                 0& {\rm id}_q\\
                 0& T_{32}^t
                 \end{matrix}
\right).$$
Let
$$J=\left(\begin{matrix}
                  J_0 &  0&0 \\
                 0& 0&{\rm id}_q\\
                 0& -{\rm id}_q&0
                 \end{matrix}
\right)$$
and $J'$ be the matrix obtained from $J$ by replacing the negative entries of it by zeroes. 
Note that $T$ and $S$ can be thought as maps from $(\mathbb{R}^n)^*$
to $\mathbb{R}^{p}\times (\mathbb{R}^p)^*\times \mathbb{R}^{q}\times (\mathbb{R}^q)^*$ (see the definition of an embedding map in Definition 2.1 of \cite{Li04}), and $S(\mathbb{Z}^n),T(\mathbb{Z}^n)\subset \mathbb{R}^{p}\times (\mathbb{R}^p)^*\times \mathbb{Z}^{q}\times (\mathbb{R}^q)^*$.
Then we can think of $S(\mathbb{Z}^n),T(\mathbb{Z}^n)$ as in $G$ via composing
$S|_{\mathbb{Z}^n},T|_{\mathbb{Z}^n}$ with the natural covering map
$\mathbb{R}^{p}\times (\mathbb{R}^p)^*\times \mathbb{Z}^{q}\times (\mathbb{R}^q)^*\rightarrow G.$ Let $P'$ and $P''$ be the canonical projections of
$G$ to $\mathcal{M}$ and $\widehat{\mathcal{M}},$ respectively, and let
$$T'=P'\circ T,\quad T''=P''\circ T,\quad S'=P'\circ S,\quad S''=P''\circ S.$$
Then the following set of formulas define a $\mathcal{B}^{\infty}$-$\mathcal{A}^{\infty}$
bimodule structure on $\mathcal{E}^{\infty}$:
\begin{empheq}[left=\empheqlbrace]{align}
  &(fU_{l}^{\theta})(x)=e^{2\pi i\langle -T(l),J'T(l)\slash2\rangle}\langle x,T''(l)\rangle f(x-T'(l)), \label{eq:module1}\\
&\langle f,g\rangle_{\mathcal{A}^{\infty}}(l)=e^{2\pi i\langle -T(l),J'T(l)\slash2\rangle}\int_{G}\langle x,-T''(l)\rangle g(x+T'(l))\overline{f(x)}dx,\label{eq:module2}\\
&(U_{l}^{\sigma(\theta)}f)(x)=e^{2\pi i\langle -S(l),J'S(l)\slash2\rangle}\langle x,-S''(l)\rangle f(x+S'(l)), \label{eq:module3}\\
& _{\mathcal{B}^{\infty}}\langle f,g\rangle(l)=Ke^{2\pi i\langle S(l),J'S(l)\slash2\rangle}\int_{G}\langle x,S''(l)\rangle \overline{g(x+S'(l))}f(x)dx, \label{eq:module4}
\end{empheq}
where $U_l^{\theta},U_{l}^{\sigma(\theta)}$ denote the canonical unitaries with respect to the group element $l\in \mathbb{Z}^n$ in $\mathcal{A}^{\infty}$ and $\mathcal{B}^{\infty}$, respectively, and $K$ is a positive constant. See Proposition 2.2 in \cite{Li04} for the following well-known result.

\begin{theorem}
  The smooth module $\mathcal{E}^{\infty}$, with above structures, is an $\mathcal{B}^{\infty}-\mathcal{A}^{\infty}$ Morita equivalence bi-module which can be completed to a strong $B-A$ Morita equivalence bi-module $\mathcal{E}.$
\end{theorem}

The completion $\mathcal{E}$ of $\mathcal{E}^{\infty}$ of the above theorem becomes a finitely generated projective module over $A$ (see the argument before Proposition 4.6 in \cite{ELPW10}). The resulting class $[\mathcal{E}] \in$ $\K_0(A)$ is called the Bott class. We will soon see that it will contribute to a generating set of $\K_0(A)$.
\begin{remark}\label{rmk:trace_of_heisenberg}
\normalfont
	The trace of the module $\mathcal{E}$, which was computed by Rieffel \cite{Rie88}, is exactly the absolute value of the pfaffian of the upper left $2p \times 2p$ corner of the matrix $\theta,$ which is $\theta_{1,1}$. Indeed, \cite[Proposition 4.3, page 289]{Rie88} says that trace of $\mathcal{E}$ is $|\deta \widetilde{T}|$, where 
$$
\widetilde{T}= \left( \begin{array}{cc}
T_{11} & 0\\
0  & \id_q\\
\end{array} \right).
$$ Thus the relation $T_{11}^tJ_0T_{11} = \theta_{1,1}$ and the fact $\deta (J_0) =1$ give the claim.
\end{remark}

Let $\theta \in \mathcal{T}_n.$ We will now see that for each non-zero pfaffian minor of $\theta,$ we can construct a projective module over $A_\theta$ such that the trace of this module is exactly the pfaffian minor.  Fix $1 \leq p \leq \frac{n}{2}$. Choose $I:=\left(i_{1}, i_{2}, \ldots, i_{2 p}\right)$ for $i_1 < i_2 < ... < i_{2p}$, and assume the pfaffian minor $\pf(\theta_I)$ is non-zero (so that $\theta_I$ is invertible). Choose a permutation $\Sigma \in \mathcal{S}_{n}$ such that $\Sigma(1) = i_1,  \Sigma(2) = i_2,\cdots, \Sigma(2p) = i_{2p}.$ If $U_1, U_2, \cdots, U_n$ are generators of $A_\theta$, there exists an $n \times n$ skew-symmetric matrix, denoted by  $\Sigma(\theta)$, such that $U_{\Sigma(1)},  U_{\Sigma(2)}, \cdots, U_{\Sigma(n)}$ are generators of $A_{\Sigma(\theta)}$ and $A_{\Sigma(\theta)} \cong A_\theta.$ Note that the upper left $2p \times 2p$ block $\Sigma(\theta)$ is exactly $\theta_{I},$ which is invertible. Now consider the projective module constructed as completion of $\mathscr{S}(\R^p \times \Z^{n-2p})$ over $A_{\Sigma(\theta)}$ as in the previous subsection and denote it by  $\mathcal{E}_{I}^{\theta}$. The trace of this module is the pfaffian of $\theta_{I}$ by the remark above, which is $\sum_{\xi \in \Pi}(-1)^{|\xi|}\prod^{p}_{s=1}\theta_{i_{\xi(2s-1)}i_{\xi(2s)}}.$  Varying $p$, and assuming that all the pfaffian minors are non-zero, we get $2^{n-1}-1$ projective modules. We call these $2^{n-1}-1$ elements the \emph{fundamental projective modules}. 

So for a non-zero $\pf(\theta_{I})$, $I=\left(i_{1}, i_{2}, \ldots, i_{2 p}\right),$ we have constructed a projective module $\mathcal{E}_{I}^{\theta}$ over $A_\theta,$ whose trace is $\pf(\theta_{I}).$ A quick thought shows that $\mathcal{E}_{I}^{\theta}$ is an equivalence bimodule between $A_\theta$ and $A_{g_{I, \Sigma}\theta}$ for some $g_{I, \Sigma} \in \mathrm{SO}(n, n|\Z).$ Indeed, let $R_{I}^{\Sigma}$ be the permutation matrix corresponding to the permutation $\Sigma$. Note that $\Sigma(\theta) = \rho\left(R_{I}^{\Sigma}\right)\theta.$    Then clearly $g_{I, \Sigma}=\sigma_{2 p}  \rho\left(R_{I}^{\Sigma}\right).$ In Section~\ref{sec:gen_general}, we will write down a basis of $\K_0(A_\theta)$ using these fundamental modules.

Next, we will construct specific (Rieffel-type) projections which represent the fundamental projective modules. The following theorem is a modification (according to our needs) of \cite[Theorem 3.13]{CH21}. From now on we shall often denote the canonical trace of $A_\theta$ by $\Tr_\theta.$ 

\begin{theorem}\label{even n projection}
  For any even number $n=2l\geq2,$ let $\theta\in \mathcal{T}_n$ be totally irrational satisfying ${\pf}(F^j(\theta)_{1,1})\in (\frac{1}{2},1)$ for $j=0,1,\dots,l-1.$ Then
  there exists a (Rieffel-type) projection $p_m$ inside $A_{F^m(\theta)}$ such that
  \begin{eqnarray}\Tr_{F^m(\theta)}(p_m)={\pf}(F^m(\theta))\nonumber
   \end{eqnarray}
   for $m=0,1,2,\dots,l-1.$ 
\end{theorem}
\begin{proof}
Since $\theta$ is totally irrational, it follows from the discussion after Lemma~\ref{h1} that ${\pf}(F^m(\theta)_{1,1})$ is irrational and
$A_{F^{m}(\theta)}$ is a simple C$^*$-algebra for $m=0,1,\dots,l-1$.
Now we do the proof using recursion on $m.$ For $m=l-1$, $F^{l-1}(\theta)$ is a $2\times2$ matrix and  ${\pf}(F^{l-1}(\theta)_{1,1})={\pf}(F^{l-1}(\theta))\in (\frac{1}{2},1)$ is irrational, construction of such projection is well-known
  (the projection is known as Rieffel projection), and the trace of this projection is ${\pf}(F^{l-1}(\theta)),$ (see Appendix~\ref{sec:bottp}). 
  Now suppose that for some $m\in \{1,2,\dots,l-1\}$, there is such a projection in
$A_{F^m(\theta)}$ such that \begin{eqnarray}\label{u1}
\Tr_{F^m(\theta)}(p_m)&=&{\pf}(F^m(\theta)).
\end{eqnarray} 
Then we want to prove that there is a projection $p_{m-1}$ in $A_{F^{m-1}(\theta)}$ with $\Tr_{F^{m-1}(\theta)}(p_{m-1})={\pf}(F^{m-1}(\theta)).$ We follow the method of page 198-199 in \cite{Boc97} to construct such a projection. Write
   $$F^{m-1}(\theta)=\left(\begin{matrix}
                  F^{m-1}(\theta)_{1,1} &  F^{m-1}(\theta)_{1,2} \\
                 F^{m-1}(\theta)_{2,1}& F^{m-1}(\theta)_{2,2}
                 \end{matrix}
\right)
\in \mathcal{T}_{n-2(m-1)},$$
where $F^{m-1}(\theta)_{1,1}$ is a $2\times 2$ block.
 From the previous discussion of this section, $A_{F^{m-1}(\theta)}$ is strong Morita equivalent to $A_{\sigma(F^{m-1}(\theta))},$
where
\begin{eqnarray}\sigma(F^{m-1}(\theta))&=&\left(\begin{matrix}
                  F^{m-1}(\theta)_{1,1}^{-1} & -F^{m-1}(\theta)_{1,1}^{-1}F^{m-1}(\theta)_{1,2} \\
                 F^{m-1}(\theta)_{2,1}F^{m-1}(\theta)_{1,1}^{-1} & F^{m-1}(\theta)_{2,2}-F^{m-1}(\theta)_{2,1}F^{m-1}(\theta)_{1,1}^{-1}F^{m-1}(\theta)_{1,2}
                \end{matrix}
\right)\nonumber\\
&=&\left(\begin{matrix}
                  F^{m-1}(\theta)_{1,1}^{-1} & -F^{m-1}(\theta)_{1,1}^{-1}F^{m-1}(\theta)_{1,2} \\
                 F^{m-1}(\theta)_{2,1}F^{m-1}(\theta)_{1,1}^{-1} & F^{m}(\theta)
                \end{matrix}
\right).\nonumber
\end{eqnarray}
Denote the Rieffel projection in $A_{F^{m-1}(\theta)_{11}}$, by $e$ and  $\Tr_{F^{m-1}(\theta)}(e)={\pf}(F^{m-1}(\theta)_{1,1})$ (see Appendix~\ref{sec:bottp} for the construction of such $e$, here we use the assumption that ${\pf}(F^{m-1}(\theta)_{1,1})$ is in $(\frac{1}{2},1)$). It follows that $A_{\sigma(F^{m-1}(\theta))}\cong eA_{F^{m-1}(\theta)}e,$ and we denote this isomorphism by $\psi$ (see Appendix~\ref{sec:bottp} for the description of $\psi$). By the induction hypothesis and Equation~\ref{u1}, there exists a projection $e'\in A_{F^{m}(\theta)}\subset A_{\sigma(F^{m-1}(\theta))}$ with
$\Tr_{F^{m}(\theta)}(e')={\pf}(F^m(\theta)).
$
 
Now for the tracial state $\Tr_{F^{m-1}(\theta)}$ on $A_{F^{m-1}(\theta)}$, since $\psi(1_{A_{\sigma(F^{m-1}(\theta))}})=e$ and $\Tr_{F^{m-1}(\theta)}(e)={\pf}(F^{m-1}(\theta)_{1,1})$, we have $$\frac{1}{{\pf}(F^{m-1}(\theta)_{1,1})}\Tr_{F^{m-1}(\theta)}\circ \psi$$ is a tracial state on $A_{\sigma(F^{m-1}(\theta))}.$ But this tracial state is $\Tr_{\sigma(F^{m-1}(\theta))}$ as $A_{\sigma(F^{m-1}(\theta))}$ has a unique tracial state (being Morita equivalent to $A_{F^{m-1}(\theta)},$ $F^{m-1}(\theta)$ is nondegenerate). Note $e'\in A_{F^{m}(\theta)}\subset A_{\sigma(F^{m-1}(\theta))},$ so
$\Tr_{\sigma(F^{m-1}(\theta))}(e')=\Tr_{F^{m}(\theta)}(e').$
Then we get
\begin{eqnarray}
&&\frac{1}{{\pf}(F^{m-1}(\theta)_{1,1})}\Tr_{F^{m-1}(\theta)}\circ \psi (e')\nonumber\\
&=&\Tr_{\sigma(F^{m-1}(\theta))}(e')=\Tr_{F^{m}(\theta)}(e')\nonumber\\
&=&{\pf}(F^m(\theta)).\label{m11}
\end{eqnarray}
 Let $p_{m-1}:=\psi(e').$
From (\ref{m11}) we get
\begin{eqnarray}
\Tr_{F^{m-1}(\theta)}(p_{m-1})&=&\Tr_{F^{m-1}(\theta)}
(\psi(e'))\nonumber\\
&=&{\pf}(F^{m-1}(\theta)_{1,1})(\frac{1}{{\pf}(F^{m-1}(\theta)_{1,1})}\Tr_{F^{m-1}(\theta)}\circ \psi (e'))\nonumber\\
&=&{\pf}(F^{m-1}(\theta)_{1,1})({\pf}(F^m(\theta))\nonumber\\
&=&{\pf}(F^{m-1}(\theta)),
\end{eqnarray}
using the equations (\ref{=2}) and (\ref{pf1}). 
\end{proof}

The above theorem tells us how to construct a higher dimensional (Rieffel-type) projection with corresponding trace values from a low dimensional projection under certain conditions. By Lemma \ref{h1}, we know that
\begin{eqnarray}{\pf}(F^j(\theta)_{1,1})=\frac{{\pf}^{\theta}_{(1,2,\dots,n-p-1,n-p)}}{{\pf}^{\theta}_{(1,2,\dots,n-p-2)}},\quad  p=n-2j-2,\quad j=1,\dots,l-1.
\end{eqnarray}
Therefore, the conditions of the above theorem can also be described as:
\begin{eqnarray}\frac{{\pf}^{\theta}_{(1,2,\dots,n-p-1,n-p)}}{{\pf}^{\theta}_{(1,2,\dots,n-p-2)}}\in (\frac{1}{2},1)\quad {\rm for}\,\, p=n-2j-2,\quad j=1,\dots,l-1
\end{eqnarray}
and $\theta_{ij} \in (\frac{1}{2},1)$ for $i<j.$ (The last condition is stronger though.)
We record this fact as a corollary below.

 \begin{corollary}\label{even n projection explicit}
  Let $\theta\in \mathcal{T}_n$ be totally irrational for $n=2l\geq2$. If $\theta$ satisfies 
  $\theta_{ij} \in (\frac{1}{2},1)$ for $i<j,$ and  \begin{eqnarray}\frac{{\pf}^{\theta}_{(1,2,\dots,n-p-1,n-p)}}{{\pf}^{\theta}_{(1,2,\dots,n-p-2)}}\in (\frac{1}{2},1)\quad {\rm for}\,\, p=n-2j-2,\quad j=1,\dots,l-1,
\end{eqnarray}  then
  there exists a (Rieffel-type) projection $p_m$ inside $A_{F^m(\theta)}$ such that
  \begin{eqnarray}\Tr_{F^m(\theta)}(p_m)={\pf}(F^m(\theta))\nonumber
   \end{eqnarray}
   for $m=0,1,2,\dots,l-1.$
\end{corollary}

In particular, when $m=0,$ we have the following:

 \begin{corollary}\label{even n projection explicit_m=0}
Let $\theta\in \mathcal{T}_n$ be totally irrational for $n=2l\geq2$. If $\theta$ satisfies $\theta_{ij} \in (\frac{1}{2},1)$ for $i<j,$ and \begin{eqnarray}\frac{{\pf}^{\theta}_{(1,2,\dots,n-p-1,n-p)}}{{\pf}^{\theta}_{(1,2,\dots,n-p-2)}}\in (\frac{1}{2},1)\quad {\rm for}\,\, p=n-2j-2,\quad j=1,\dots,l-1,
\end{eqnarray}  then
  there exists a (Rieffel-type) projection $p=p_0$ inside $A_{\theta}$ such that
$\Tr(p)={\pf}(\theta).$
  
\end{corollary}

 Let us now show that under suitable conditions on $\theta,$ the numbers coming in the right hand side of Equation~\ref{range of trace} may be realized as traces of the Rieffel-type projections in $A_\theta.$
\begin{theorem}\label{PI1}
Let $\theta\in \mathcal{T}_n$ be totally irrational. If for some $I \in \mathrm{Minor}(n)$ with $|I|=2m\geq 2$, 
$ {\pf}(F^{j}(\theta_I)_{1,1})\in (\frac{1}{2},1),
$
 for all $j=0,1,\dots,m-1,$
  then there exist  (Rieffel-type) projections $P_{I},$ such that
  \begin{eqnarray}\label{pI1}\Tr(P_{I})={\pf}(\m{\theta}{I})
  \end{eqnarray}
where $\Tr$ is the canonical tracial state on $A_{\theta}.$
\end{theorem}
\begin{proof}
  We use $\m{\theta}{I}$ instead of $\theta$ in Corollary~\ref{even n projection explicit_m=0}, noting that $\m{\theta}{I}$ is totally irrational as well, to get $\Tr(P_{I})={\pf}(\m{\theta}{I}).$
\end{proof}

We will now see that the above projections form a generating set of $\K_0(A_\theta),$ when $\theta$ is totally irrational, under the assumptions of the above theorem.

\begin{definition}\label{a3}
 We say that $\theta$ is {\it strongly totally irrational} if $\theta=(\theta_{jk})\in \mathcal{T}_n$ is a totally irrational matrix such that
 \begin{eqnarray}\label{a2}
 {\pf}(F^{j}(\m{\theta}{I})_{1,1})\in (\frac{1}{2},1) 
 \end{eqnarray}
for all $I \in \mathrm{Minor}(n)$ with $|I|=2m\geq 2$,  and for all $j=0,1,\dots,m-1.$ 
\end{definition}
We refer to Appendix~\ref{sec:strongly_irr} for more on strongly totally irrational matrices and for the construction of examples of such matrices.
\begin{theorem}\label{PI}
  For any integer $n\geq 2,$ let $\theta\in \mathcal{T}_n$ be strongly totally irrational. Then there exist  (Rieffel-type) projections $P_{I},$ for every $I \in \mathrm{Minor}(n)$ inside $A_{\theta},$ such that
 $\label{pI1}\Tr(P_{I})={\pf}(\m{\theta}{I}),$
  where $\Tr$ is the canonical tracial state on $A_{\theta}.$  Moreover, a generating set of $\K_0(A_{\theta})$ is given by  $\{[P_{I}]~|~I \in \mathrm{Minor}(n)\}.$ 
\end{theorem}
\begin{proof} For $I\neq \emptyset,$ by Theorem \ref{PI1} and from the definition of strong total irrationality,  we know that those projections $\{P_{I}\}$ exist. For $I= \emptyset,$  $P_{I}=1$ by definition.
  Now since $\theta$ is totally irrational, $\Tr$ is injective. So $\{[P_{I}]~|~I \in \mathrm{Minor}(n)\}$ are the generators of $\K_0(A_{\theta})$ by Equation~\ref{range of trace}.
\end{proof}

\subsection{Pimsner--Voiculescu exact sequence and the Rieffel-type projections}\label{sec:pv_generators}	
Recall that for crossed products like $A \rtimes_\gamma \Z$, the Pimsner--Voiculescu sequence looks like 
{\[
\begin{CD}
\K_0 (A) @>{\id -\gamma_{*}^{-1}}>>\K_0 (A)@>{i_*}>> \K_0 (A\rtimes \Z)   \\
@A{}AA & &  @VV{e_2}V            \\
 \K_1 (A\rtimes \Z) @<<{i*}<  \K_1 (A) @<<{\id -\gamma_{*}^{-1}}< \K_1 (A)
\end{CD}
\]} 
where $i$ is the inclusion.

Let $\theta\in \mathcal{T}_n$ and $\theta'$ be  the  upper left $(n-1) \times (n-1)$ block of $\theta.$
Let $\gamma$ be the automorphism on $A_{\theta'}$ given by
$\gamma (U_i) = e(-\theta_{in}) U_i$ for $ i=1,\dots,(n-1),$ as in Section~\ref{sec:definition}. For the crossed product algebra $A_{\theta'}\rtimes_\gamma \Z\cong A_{\theta},$
the Pimsner--Voiculescu sequence becomes

\[
\begin{CD}
\K_0(A_{\theta'}) @>{\id -\gamma_{*}^{-1}}>>\K_0(A_{\theta'})@>{i_*}>> \K_0 (A_\theta)   \\
@A{}AA & &  @VV{e_2}V            \\
 \K_1(A_\theta)@<<{}<  \K_1(A_{\theta'}) @<<{\id -\gamma_{*}^{-1}}< \K_1(A_{\theta'})
\end{CD}
\]
Since $\gamma$ is homotopic to the identity map(\cite[Lemma 1.5]{Phi06}), $\id -\gamma_{*}^{-1}$
is the zero map. Hence we get the following exact sequences

\begin{center}
\begin{equation}\label{eq:pimsner_k_0}
	\begin{tikzcd}
  0 \arrow{r} & \K_0(A_{\theta'}) \arrow[r, "i_*"] & \K_0(A_{\theta}) \arrow[r, "e_2"] & \K_1(A_{\theta'})          \arrow{r} & 0 
    \end{tikzcd}
    \end{equation}
    	\end{center}
    	
\begin{center}
\begin{equation}\label{eq:pimsner_k_1}
	\begin{tikzcd}
  0 \arrow{r} & \K_1(A_{\theta'}) \arrow{r} & \K_1(A_{\theta}) \arrow{r}& \K_0(A_{\theta'})          \arrow{r} & 0 
    \end{tikzcd}
    \end{equation}
    	\end{center}
From these two exact sequences (along with the fact $\K_0(C(\T)) = \K_1(C(\T)) = \Z$,) by induction on $n$ we get
\begin{equation}
\K_0(A_\theta)\cong \Z^{2^{n-1}} \cong \K_1(A_\theta).
\end{equation}
Now we want to fit the Rieffel-type projections in Equation~\ref{eq:pimsner_k_0}. In order to achieve that, we again assume $\theta$ to be strongly totally irrational. Now we know from Theorem~\ref{PI} that for different $I:=\left(i_{1}, i_{2}, \ldots, i_{2 p}\right)$ such that $ I\in \mathrm{Minor}(n),$ the $\K$-theory classes of the Rieffel-type projections $P_{I}^{\theta}:=P_{I}$  generate $\K_0(A_\theta).$ Now we claim that for $I\in \mathrm{Minor}(n-1),$ 

\begin{equation}\label{eq:i_*_image}
	i_*([P_{I}^{\theta'}]) = [P_{I}^{\theta}].
\end{equation}
  This follows from the fact that they have the same trace in $A_\theta,$ which is $\pf(\m{\theta}{I}).$ Now if $ \widetilde{\mathrm{Minor}}(n)$ denotes the set of all  $I\in \mathrm{Minor}(n)$ such that $i_{2p}=n,$ i.e. $ \widetilde{\mathrm{Minor}}(n)= \mathrm{Minor}(n)\setminus \mathrm{Minor}(n-1),$  the collection $\{[P_{I}^{\theta}]\}_{I\in\widetilde{\mathrm{Minor}}(n)}$ maps via $e_2$ to a generating set of $\K_1(A_{\theta'})$ which follows from the fact that Equation~\ref{eq:pimsner_k_0} is exact. We record these observations in the following proposition.
  
 \begin{proposition}\label{prp:pv_image}
 
 For a strongly totally irrational $\theta,$ $i_*([P_{I}^{\theta'}]) = [P_{I}^{\theta}],$ for $ I\in \mathrm{Minor}(n-1),$  and $\{e_2([P_{I}^{\theta}])\}_{I\in \widetilde{\mathrm{Minor}}(n)}$ form a generating set of $\K_1(A_{\theta'}).$ 
 	
 \end{proposition}


	\section{K-theory of $A_\theta\rtimes \Z_2$}\label{sec:k-theory}
		
		Before computing the K-theory groups of 	$A_\theta\rtimes \Z_2$ we will see how the Rieffel-type projections give rise to $\K_0$-classes of $A_\theta\rtimes \Z_2.$ We start with the following well known facts.

	 \begin{proposition}\label{prop:mainprop}
	Suppose $F$ is a finite group acting on a C$^*$-algebra $A$ by the action $\alpha$. Also suppose that $\mathcal{E}$  is a finitely generated projective (right) $A$-module with a right action $T : F \rightarrow \Aut(\mathcal{E})$, written $(\xi, g) \rightarrowtail \xi T_g$, such that $\xi (T_g)a = (\xi\alpha_g(a)) T_g$ for all $\xi \in  \mathcal{E}, a \in A,$ and $g \in F$. Then $\mathcal{E}$ becomes a finitely generated projective $A\rtimes F$ module with action defined by
	
$$ \xi \cdot (\sum_{g \in F} a_g\delta_g) =  \sum_{g \in F} (\xi a_g)T_g.$$ 

Also, if we restrict the new module to $A$, we get the original $A$-module $\mathcal{E}$, with the action of $F$ forgotten. 
\end{proposition}

\begin{proof}
	This is exactly the construction of the Green--Julg map, which is a map from $\K_0^F(A)$, the $F$-equivariant K-theory of $A,$ to $\K_0(A\rtimes F)$ . See \cite[Proposition 4.5]{ELPW10}.
\end{proof} 

For a general crossed product $A \rtimes \Z_2$, for an action $\beta$ of $\Z_2$ on $A$, the above Green--Julg map is easy to describe for $\Z_2$-invariant projections in $A$. If $P$ is a $\Z_2$-invariant projection in $A$, the corresponding projection in $A \rtimes \Z_2$ is $\frac{P}{2}(1+W),$ where $W$ denotes the canonical non-trivial unitary of $\Z_2$ in $A \rtimes \Z_2.$
Also define the natural map (regular representation) $p$ which goes from $A \rtimes \Z_2$ to $\M_2(A)$  such that 
\begin{equation}\label{eq:def_of_p}
	 p(a + bW) = \left( \begin{array}{ccc}
a & b\\
\beta_g(b) & \beta_g(a)\\
  \end{array} \right),\quad a,b\in A\end{equation} where $g$ is the non-trivial element of $\Z_2.$ This induces a map $p_*: \K_0(A \rtimes \Z_2) \rightarrow \K_0(A),$ which is known to be the inverse of the Green--Julg map (see \cite[page 191]{HG04}).
  
  With the above facts in hand, we start with the projective modules over $A_\theta$ which were described in the previous section. Recall that the projective module $\mathcal{E}$ is a completion of the space  $\mathscr{S}(\R^p \times \Z^q)$, for some  $p, q \in \Ne$ such that  $n=2p + q$. For the flip action of $\Z_2$ on $A_\theta$, we define the following action, again called the flip action, of $\Z_2$ on the dense subspace $\mathscr{S}(\R^p \times \Z^q)$ of $\mathcal{E}$ by

  \begin{equation}\label{eq:flip_action_module}
	T_g (f)(x,t):= f(-x,-t),
	\end{equation}
	where $g$ is the non-trivial element of $\Z_2.$ 
	
	Using Equation~\ref{eq:module1} and  Equation~\ref{eq:module2} it is quickly checked that $\Z_2$ defines an action on $\mathcal{E}$ which is compatible with the flip action of $\Z_2$ on $A_\theta$ in the sense of Proposition~\ref{prop:mainprop} (cf. \cite[Section 7]{CL19}, \cite[Section 3]{Cha21}). In particular, we have, 
	\begin{equation}\label{eq:proj_mod_inner_holds}
  \langle fT_g,f'T_g \rangle_{\mathcal{A}^\infty}=\beta_g(\langle f,f' \rangle_{\mathcal{A}^\infty}),
\end{equation} and 
\begin{equation}\label{eq:proj_mod_T_holds}
f (T_g)a = (f\beta_g(a)) T_g,
\end{equation} 
for $f$ and $f'\in \mathscr{S}(\R^p \times \Z^q),$ $a\in A_\theta$ (see \cite[Equation 3.9, Equation 3.11]{Cha21}). Similarly, the same set of equations holds true for the left  $A_{\sigma(\theta)}$-module $\mathcal{E},$ where $\sigma(\theta)$ is as in Equation~\ref{eq:defoftheta'}. Hence $\mathcal{E}$ becomes a projective module over the crossed product $A_\theta\rtimes \Z_2.$ We call this module $\tilde{\mathcal{E}}.$ Let $\Tr^{\Z_2}_\theta$ denote the canonical trace on $A_\theta\rtimes \Z_2$ defined by  $\Tr^{\Z_2}_\theta(a+bW)=\Tr_\theta(a)$ for $a,b\in A_\theta.$ From \cite[Lemma 4.1]{Cha21} we can compute the trace of the K-theory class of $\tilde{\mathcal{E}}$ as

\begin{equation}\label{eq:trace_of_extended_module}
	\Tr^{\Z_2}_\theta([\tilde{\mathcal{E}}])=\frac{\Tr_\theta(([\mathcal{E}])}{2}.
\end{equation}

	 Our next step is to understand how the Rieffel-type projections $P_{I}^{\theta}$ give various projections in $A_\theta\rtimes \Z_2.$ First let us introduce some notations. For $I=(i_1,i_2,\dots,i_{2p}) \in \mathrm{Minor}(n)\setminus \{\emptyset\},$ define $I^c:=(i_{2p}+1,i_{2p}+2,\dots,n),$ and for $I=\emptyset,$ define  $I^c:=(1,2,\dots,n).$  Also regarding $I^c$ as a finite sequence, by $J\subseteq I^c$ we mean a finite subsequence of $I^c,$ with the understanding that $J$ can be the empty sequence too. Finally for $J=(j_1,j_2,\dots,j_q)\subseteq I^c,$ define $U_J:=U_{j_1}U_{j_2}\cdots U_{j_q}\in A_\theta$, and for $J=\emptyset,$ $U_J:=1.$ The length $|I^c|$ of $I^c$ (or any sub-sequence of $I^c$) is defined as the number of elements in $I^c.$

	  From Appendix~\ref{sec:bottpincrossed}, we know that all the $P_{I}^{\theta}$'s are $\Z_2$-invariant. (This follows from the fact that in Theorem~\ref{even n projection}, $\psi$ is $\Z_2$-equivariant and $e$ is $\Z_2$-invariant.) Now let us see how these give rise to different projections in $A_\theta\rtimes \Z_2.$ For each $I\in \mathrm{Minor}(n),$  fix a $J\subseteq I^c$ as before. let $\bar{J}$ denote $(j_q,j_{q-1},\dots,j_1)$ for $J=(j_1,j_2,\dots,j_q).$ let $r_J$ be the number such that $U_J=e(-2r_J)U_{\bar{J}}.$ Of course, $r_J$ is a number involving $\theta_{j_{l}j_{m}}, 1\le l< m\le q.$  Now one quickly checks that $e(r_J)U_JW=:W_J$ is a self-adjoint unitary in $A_\theta\rtimes \Z_2.$ If $I=(i_1,i_2,\dots,i_{2p}),$ then $W_JU_{i_k}W_J=e(r_{J,i_k})U_{i_k}^{-1},$ for some number $r_{J,i_k}$ involving $\theta_{i_kj_{l}}, 1\le l\le q.$ As in Lemma~\ref{lemma:flip=flip}, set $\widetilde{U_{i_k}}=e(-\frac{r_{J,i_k}}{2})U_{i_k},$ and we have  $W_J\widetilde{U_{i_k}}W_J=\widetilde{U_{i_k}}^{-1}$ such that $A_{\theta_I}$ is generated by $\{\widetilde{U_{i_k}}\}_k,$ and $A_{\theta_I}\rtimes \Z_2$ sits canonically inside $A_\theta\rtimes \Z_2.$ Here in the crossed product $A_{\theta_I}\rtimes \Z_2,$ $\Z_2$ acts by the flip action and the generator of $\Z_2$ is identified with $W_J$ inside $A_\theta\rtimes \Z_2.$ Now we can construct $P_{I,J}^{\theta}:=\frac{P_I^{\theta}}{2}(1+W_J)\in A_{\theta_I}\rtimes \Z_2 \subseteq A_\theta\rtimes \Z_2,$ since $P_I^{\theta}$ is flip invariant projection inside $A_{\theta_I}$. For $I=\emptyset,$ $P_{I,J}^{\theta}:=\frac{1}{2}(1+W_J).$ 
	  Hence for each $I\in \mathrm{Minor}(n),$ and each $J\subseteq I^c,$ we have constructed a projections inside $A_\theta\rtimes \Z_2.$ Varying $I,$ and varying $J,$ we get a family of projections $\{P_{I,J}^{\theta}\}$ in $A_\theta\rtimes \Z_2.$ 
	  
	  Now we claim that we have exactly $3\cdot2^{n-1}-1$ projections $\{P_{I,J}^{\theta}\}$ if we restrict to $|J|\le 2,$ i.e the set 
\begin{equation}\label{eq:def_of_P_n}
	\mathcal{P}_n:= \underset{I\in \mathrm{Minor}(n)}{\cup}\left\{P_{I,J}^{\theta}~\suchthat ~J\subseteq I^c,|J|\le 2 \right\}\end{equation} has $3\cdot2^{n-1}-1$ many elements. This can be shown using a simple induction argument. The statement holds for $n=2, 3$ simply by counting. Now assume $|\mathcal{P}_n|= 3\cdot2^{n-1}-1,$ and we need to show that $|\mathcal{P}_{n+1}|= 3\cdot2^{n}-1.$ If $I\in \mathrm{Minor}(n),$ of course, $I\in \mathrm{Minor}(n+1),$ and if $J \subseteq I^c$, if we view $I$ as an element of $\operatorname{Minor}(n)$, then we also have $J \subseteq I^c$ if we view $I$ as an element of $\operatorname{Minor}(n+1)$. Hence we may regard $\mathcal{P}_n$ as a subset of $\mathcal{P}_{n+1}$. Thus we only need to show that there are $3 \cdot 2^n-3 \cdot 2^{n-1}=3 \cdot 2^{n-1}=2^n+2^{n-1}$ extra elements in $\mathcal{P}_{n+1} \backslash \mathcal{P}_n$. These extra elements are described below.

\begin{itemize}
	\item for $I\in \mathrm{Minor}(n),$ if we view $I$ inside  $\mathrm{Minor}(n+1),$ we have elements $P_{I,J}^{\theta},$ where $J=(n+1);$ and for $I\in \widetilde{\mathrm{Minor}}(n+1),$ we have elements $P_{I,J}^{\theta},$ where $J=\emptyset.$ This way we get $2^n$ number of elements;
	\item for any $I\in \widetilde{\mathrm{Minor}}(j)$ with $2\le j \le n-1,$ we have $(n-j)$ number of elements $P_{I,J}^{\theta},$ where $J=(j+1,n+1), (j+2,n+1), \ldots ,(n,n+1).$ This way we get 
	$2^{n-1}-n$ number of elements;
	\item Finally for $I=\emptyset,$ we have $n$ number of elements $P_{I,J}^{\theta},$ where $J=(1,n+1), (2,n+1), \ldots ,(n,n+1).$ 
	\end{itemize}    
Combining the above extra elements, we get our claim. Shortly we will see that the K-theory classes of the projections in $\mathcal{P}_n$ (along with the element $[1]$)  generate $\K_0(A_\theta\rtimes \Z_2).$

	Let us recall the exact sequence due to Lance and Natsume. For $A_\theta\rtimes \Z_2 \cong A_{\theta'}\rtimes_\phi \Z_2 *  \Z_2$ (see Section~\ref{sec:definition}), we have the following exact sequence (\cite{Nat85}).
	
	{\small \[
\begin{CD}
\K_0 (A_{\theta'}) @>{i_{1*}-i_{2*}}>>\K_0 (A_{\theta'}\rtimes_\alpha  \Z_2)\oplus \K_0 (A_{\theta'}\rtimes \Z_2)@>{j_{1*}+j_{2*}}>> \K_0 (A_{\theta'}\rtimes_\phi \Z_2*\Z_2)   \\
@A{}AA & &  @VV{e_1}V            \\
 \K_1 (A_{\theta'}\rtimes_\phi \Z_2*\Z_2) @<<{j_{1*}+j_{2*}}<  \K_1 (A_{\theta'}\rtimes \Z_2)\oplus \K_1 (A_{\theta'}\rtimes \Z_2) @<<{i_{1*}-i_{2*}}< \K_1 (A_{\theta'})
\end{CD}
\]}where $i_1, i_2, j_1,j_2$ are natural inclusions.          
	Also $A_{\theta'}\rtimes_\phi \Z_2 *  \Z_2$ is isomorphic to 	$(A_{\theta'}\rtimes_\gamma \Z)\rtimes\Z_2,$ where $\Z_2$ acts on $A_{\theta'}$ by the flip action and on the group $\Z$ by $x\rightarrow -x.$ Before computing $\K_0 (A_{\theta'}\rtimes_\phi \Z_2*\Z_2),$ we will first explicitly describe the maps in the above exact sequence.

	\subsection*{The map $e_1$}
	As an immediate corollary of \cite[Theorem 7.1]{CL19}, we have the following.
	
	\begin{proposition}\label{prop:main1}
	The diagram 
\begin{equation*}
\xymatrixrowsep{5pc}
\xymatrixcolsep{5pc}
\xymatrix
{
\K_0 (A_{\theta'}\rtimes_\phi \Z_2*\Z_2) \ar[r]^{e_1} \ar[rd]^{p_*} &
\K_1 (A_{\theta'}) \\
&\K_0(A_{\theta'}\rtimes_\gamma \Z),\ar[u]^{e_2}}
\end{equation*}
is commutative, where $p_*$ is the map induced by the natural map (see Equation~\ref{eq:def_of_p}) $p:A_{\theta'}\rtimes_\phi \Z_2 *  \Z_2 \cong (A_{\theta'}\rtimes_\gamma \Z)\rtimes\Z_2 \rightarrow \M_2(A_{\theta'}\rtimes_\gamma \Z).$ 
\end{proposition}

\begin{proof}
	Immediate from \cite[Theorem 7.1]{CL19}.
\end{proof}

Now from Proposition~\ref{prp:pv_image}, $\{e_2([P_{I}^{\theta}])\}_{I\in \widetilde{\mathrm{Minor}}(n)}$ form a generating set of $\K_1(A_{\theta'}),$ for a strongly totally irrational $\theta.$  But for $I\in \widetilde{\mathrm{Minor}}(n)$, $p_*([P_{I,\emptyset}^{\theta}]) = p_*([\frac{P_{I}^{\theta}}{2}(1+W)]) = [P_{I}^{\theta}],$ since $p_*$ acts as the inverse of the Green--Julg map in Proposition~\ref{prop:mainprop}. So we have the following.

\begin{corollary}\label{cor:gen_boundary}
	For a strongly totally irrational $\theta,$ $\{e_1([P_{I,\emptyset}^{\theta}])\}_{I\in \widetilde{\mathrm{Minor}}(n)}$ form a generating set of $\K_1(A_{\theta'}).$
\end{corollary}

\begin{proof}
	Use Proposition~\ref{prop:main1}.
\end{proof}

 \subsection*{The map $j_{1*}+j_{2*}$}  
 Next we want to understand the map $$j_{1*}+j_{2*} : \K_0 (A_{\theta'}\rtimes_\alpha  \Z_2)\oplus \K_0 (A_{\theta'}\rtimes \Z_2)\longrightarrow \K_0 (A_{\theta'}\rtimes_\phi \Z_2*\Z_2),$$ for the elements in $\mathcal{P}_{n-1}.$
Since both maps $j_{1},j_{2}$ are inclusion maps, and $A_{\theta'}\rtimes_\alpha  \Z_2$ is isomorphic to $A_{\theta'}\rtimes \Z_2$ by Lemma~\ref{lemma:flip=flip}, it is enough to work with $j_{2*}.$ Now $j_{2*}$ is induced from the natural inclusion map 
$$j_{2} : A_{\theta'}\rtimes \Z_2\longrightarrow A_{\theta}\rtimes \Z_2.$$
 We now have the following proposition.

\begin{proposition}\label{prop:gen_inclusion}

For a strongly totally irrational $\theta,$ $j_{2*}([P_{I,J}^{\theta'}]) = [P_{I,J}^{\theta}],$ for $I\in \mathrm{Minor}(n-1), J\subseteq I^c.$
\end{proposition}

 \begin{proof}
    This is clear from Equation~\ref{eq:i_*_image} as the map $j_2$ respects the inclusion map of $A_{\theta'}$ to $A_{\theta},$ and respects the unitaries  coming from $\Z_2.$ 	
	 \end{proof}
	
	 \subsection*{The map $i_{1*}-i_{2*}$} 
 We start with the following lemma.
	 
	 \begin{lemma}\label{lemma:i_is_injective}
$$i_{2*} : \K_0(A_{\theta})\longrightarrow \K_0(A_{\theta}\rtimes \Z_2)$$	 \end{lemma} \noindent is injective when $\theta$ is strongly totally irrational. 

\begin{proof}
If $\Tr_\theta^{\Z_2}$ denotes the canonical tracial state on $A_{\theta}\rtimes \Z_2,$ we have  $\Tr_\theta^{\Z_2}(i_{2*}([P_{I}^{\theta}])= \pf(\m{\theta}{I}).$ Recall that $\{[P_I^\theta]|~ I \in \mathrm{Minor}(n)\},$ generate $\K_0(A_{\theta}).$ Now if 	$i_{2*}([X])= 0,$ writing  $[X]= \sum r_I [P_{I}^{\theta}]$ for  $r_I \in \Z,$ we have $i_{2*}(\sum r_I [P_I^{\theta}])=0.$ Since $\pf(\m{\theta}{I})$ are rationally independent, taking $\Tr_\theta^{\Z_2}$ of the expression $i_{2*}(\sum r_I [P_{I}^{\theta}])$ gives $r_I=0.$
\end{proof}

\begin{corollary}\label{cor:K_1_zero}
	$\K_1(A_{\theta}\rtimes \Z_2)=0,$ when $\theta$ is strongly totally irrational. 
\end{corollary}

\begin{proof}
	
	From the left side of the diagram
	
		{\small \[
\begin{CD}
\K_0 (A_{\theta'}) @>{i_{1*}-i_{2*}}>>\K_0 (A_{\theta'}\rtimes_\alpha  \Z_2)\oplus \K_0 (A_{\theta'}\rtimes \Z_2)@>{j_{1*}+j_{2*}}>> \K_0 (A_{\theta'}\rtimes_\phi \Z_2*\Z_2)   \\
@A{}AA & &  @VV{e_1}V            \\
 \K_1 (A_{\theta'}\rtimes_\phi \Z_2*\Z_2) @<<{j_{1*}+j_{2*}}<  \K_1 (A_{\theta'}\rtimes \Z_2)\oplus \K_1 (A_{\theta'}\rtimes \Z_2) @<<{i_{1*}-i_{2*}}< \K_1 (A_{\theta'})
\end{CD}
\]}  we compute $\K_1(A_{\theta}\rtimes \Z_2)=  \K_1 (A_{\theta'}\rtimes_\phi \Z_2*\Z_2)$ by induction on $n$. Since $i_{1*}-i_{2*}$ is injective from Lemma~\ref{lemma:i_is_injective}, and we know $\K_1(A_{\theta}\rtimes \Z_2)=0$ for the low-dimensional cases, the result follows. 
\end{proof}

	Next we want to understand the map $$i_{1*}-i_{2*} :\K_0 (A_{\theta'}) \longrightarrow  \K_0 (A_{\theta'}\rtimes_\alpha  \Z_2)\oplus \K_0 (A_{\theta'}\rtimes \Z_2) ,$$ for a Rieffel-type projection $P_I,$ $I\in \mathrm{Minor}(n-1).$
Since both maps $i_{1*},i_{2*}$ are inclusion maps, and $A_{\theta'}\rtimes_\alpha  \Z_2$ is isomorphic to $A_{\theta'}\rtimes \Z_2$ by Lemma~\ref{lemma:flip=flip}, it is enough to work with $i_{2*}.$ Note that $i_{2*}$ is induced from the natural inclusion map 
$$i_{2} : A_{\theta'}\longrightarrow  A_{\theta'}\rtimes \Z_2.$$
 We now have the following proposition.

\begin{proposition}\label{prop:imageofi_rieffel} Let $\theta$ be an $n \times n$ strongly totally irrational matrix and $i_{2*} : \K_0(A_{\theta})\longrightarrow \K_0(A_{\theta}\rtimes \Z_2)$ be induced by the canonical inclusion map $i_2.$ Then for $I\in \mathrm{Minor}(n)\setminus \{\emptyset\},$
$$i_{2*}([P_I^{\theta}])=	2[P_{I,\emptyset} ^{\theta}]-[P_{I'',\emptyset} ^{\theta}]+[P_{I'',(i_{2p-1})} ^{\theta}]-[P_{I'', (i_{2p})} ^{\theta}]+[P_{I'', (i_{2p-1},i_{2p})} ^{\theta}],$$ where $I=(i_1,i_2,\dots,i_{2p})$ and $I''$ is obtained from $I$ by deleting the last two numbers.
\end{proposition}

We would like to prove the above proposition by induction on the length $|I|$ of $I$. Before we go to the proof, we explain the two-dimensional case in the lemma below. 

		\begin{lemma}\label{lemma:2dim-result}(cf. the proof of \cite[Corollary 7.2 ]{CL19}) Let $\theta_{12}$ be an irrational number in $(\frac{1}{2},1)$ and $i_{2*} : \K_0(A_{\theta_{12}})\longrightarrow \K_0(A_{\theta_{12}}\rtimes \Z_2)$ be induced by the canonical inclusion map $i_2.$ Then

$$i_{2*}([P_{(1,2)}^{\theta_{12}}])=	2[P_{(1,2),\emptyset} ^{\theta_{12}}]-[P_{\emptyset,\emptyset} ^{\theta_{12}}]+[P_{\emptyset,(1)} ^{\theta_{12}}]-[P_{\emptyset, (2)} ^{\theta_{12}}]+[P_{\emptyset, (1,2)} ^{\theta_{12}}].$$
	\end{lemma}
	
		\begin{proof}
		In both cases, the two K-theory elements in RHS and in LHS have the same vector trace $(\theta;0,0,0,0)$ (\cite[Page 597]{Wal95}). Then the result follows from \cite[Corollary 5.6]{Wal95}.
	\end{proof}
	
	\begin{proof}[Proof of Proposition~\ref{prop:imageofi_rieffel}]
	 As we already mentioned, we will prove this by induction on the length of $I$, $|I|.$ For any $I\in \mathrm{Minor}(n)\setminus \{\emptyset\}$ we first have the following commutative diagram:	
	 
	 \begin{align*}
\xymatrixrowsep{5pc}
\xymatrixcolsep{5pc}
\xymatrix
{
\K_0(A_{\theta_I}) \ar[r]^{i_{2*}} \ar[d]^{i_*} &
\K_0(A_{\theta_I}\rtimes\Z_2) \ar[d]^{j_{2*}} \\
\K_0(A_{\theta})\ar[r]^{i_{2*}} & \K_0(A_{\theta}\rtimes\Z_2)
}
\end{align*}
where all the maps are induced by inclusions.
For $|I|=2,$ $I=(i_1,i_2),$ using Lemma~\ref{lemma:2dim-result} and using the above commutative diagram (along with Proposition~\ref{prop:gen_inclusion}), we indeed get 
$$i_{2*}([P_{(i_1,i_2)}^{\theta}])=	2[P_{(i_1,i_2),\emptyset} ^{\theta}]-[P_{\emptyset,\emptyset} ^{\theta}]+[P_{\emptyset,(i_1)} ^{\theta}]-[P_{\emptyset, (i_2)} ^{\theta}]+[P_{\emptyset, (i_1,i_2)} ^{\theta}].$$

Now for the induction step, assume that the statement is true for any $I$ with $|I|< 2p$. Then we will show that the statement is true for an $I$ with $|I|= 2p.$ Assume $I=(i_1,i_2,\dots,i_{2p}).$ Due to the commutative diagram, it is enough to prove the statement viewing $i_{2*}$ as a map from $\K_0(A_{\theta_I})$ to $\K_0(A_{\theta_I}\rtimes\Z_2).$ Hence we can assume $I= (1,2,\dots,2p),$ and $\theta= \theta_I,$ and we want to show 
$$i_{2*}([P_{(1,2,\dots,2p)}^{\theta}])=	2[P_{I,\emptyset} ^{\theta}]-[P_{I'',\emptyset} ^{\theta}]+[P_{I'',({2p-1})} ^{\theta}]-[P_{I'', ({2p})} ^{\theta}]+[P_{I'', ({2p-1},{2p})} ^{\theta}].$$

From Appendix~\ref{sec:bottpincrossed} we have the following commutative diagrams. 

 \begin{multicols}{2}
\xymatrixrowsep{5pc}
\xymatrixcolsep{5pc}
\xymatrix
{
A_{\sigma({\theta})} \ar[r]^{i_2} \ar[d]^{\psi} &
A_{\sigma({\theta})}\rtimes\Z_2 \ar[d]^{\psi} \\
e A_{\theta}e \ar[r]^{i_2} \ar[d]^{i}&
eA_{{\theta}}e\rtimes \Z_2 \ar[d]^{i}\\
 A_{\theta} \ar[r]^{i_2} &
A_{\theta}\rtimes \Z_2 
}
\xymatrixrowsep{5pc}
\xymatrixcolsep{5pc}
\xymatrix
{
\K_0(A_{\sigma({\theta})}) \ar[r]^{i_{2*}} \ar[d]^{\psi_*} &
\K_0(A_{\sigma({\theta})}\rtimes\Z_2) \ar[d]^{\psi_*} \\
\K_0(e A_{\theta}e) \ar[r]^{i_{2*}}\ar[d]^{i*} &
\K_0(eA_{{\theta}}e\rtimes \Z_2)\ar[d]^{i*}\\
\K_0(A_{\theta}) \ar[r]^{i_{2*}} &
\K_0(A_{\theta}\rtimes \Z_2) 
}
\end{multicols}
where
\begin{equation*}
	\sigma(\theta):=\left(\begin{matrix}
                  \theta_{1,1}^{-1} & -\theta_{1,1}^{-1}\theta_{1,2} \\
                  \theta_{2,1}\theta_{1,1}^{-1} & \theta_{2,2}-\theta_{2,1}\theta_{1,1}^{-1}\theta_{1,2}
                \end{matrix}
\right) \end{equation*} for $$ \theta = \left(\begin{matrix}
                  \theta_{1,1} & \theta_{1,2} \\
                  \theta_{2,1} & \theta_{2,2}
                \end{matrix}
\right).$$ Here, as before $\theta_{1,1}$ is a $2\times 2$ matrix and $\psi$ as in the proof of Theorem~\ref{even n projection}. Let $V_1,V_2, \ldots, V_{2p}$ be the generators of $A_{\sigma(\theta)}$ as in Appendix~\ref{sec:riep}. As in the previous section, 
let $F(\theta)=
                 \theta_{2,2}-\theta_{2,1}\theta_{1,1}^{-1}\theta_{1,2}\in \mathcal{T}_{n-2}.$ Then the generators of $A_{F(\theta)}\subset A_{\sigma(\theta)}$ are $V_3,V_4, \ldots, V_{2p}.$ Take $J=(1,2,\ldots, 2p-2).$ Then by the induction step, we have

                 $$i_{2*}([P_{J}^{F(\theta)}])=	2[P_{J,\emptyset} ^{F(\theta)}]-[P_{J'',\emptyset} ^{F(\theta)}]+[P_{J'',({2p-3})} ^{F(\theta)}]-[P_{J'', ({2p-2})} ^{F(\theta)}]+[P_{J'', ({2p-3},{2p-2})} ^{F(\theta)}].$$
                From the proof of Theorem~\ref{even n projection}, we know that 
$\psi(P_{J}^{F(\theta)})=P_{I}^{\theta}.$ It is clear that $\psi(P_{J,\emptyset}^{F(\theta)})=P_{I,\emptyset}^{\theta}.$ Now we want to show that $\psi_*([P_{J'',(2p-3)}^{F(\theta)}])=[P_{I'',(2p-1)}^{\theta}].$ By definition, $P_{J'',(2p-3)}^{F(\theta)}=\frac{P_{J''}^{F(\theta)}}{2}(1+V_{2p-1}W).$ Now $\psi(P_{J''}^{F(\theta)})=P_{I''}^{\theta}\in eA_\theta e.$ $\psi(P_{J'',(2p-3)}^{F(\theta)})= \frac{P_{I''}^{\theta}}{2}(e+\psi(V_{2p-1})W).$ But with the computation of $\frac{1}{2}(e+\psi(V_{2p-1})W)$ in Appendix~\ref{sec:bottpincrossed}, along with the arguments at the end of Appendix~\ref{sec:bottpincrossed} shows that $\frac{P_{I''}^{\theta}}{2}(e+\psi(V_{2p-1})W)$ is homotopic to $P_{I''}^{\theta}\cdot\frac{e}{2}(1+U_{2p-1}W)$ in $A_\theta\rtimes \Z_2.$ Since  $P_{I''}^{\theta}\in eA_\theta e,$ we indeed have $\psi_*([P_{J'',(2p-3)}^{F(\theta)}])=[P_{I'',(2p-1)}^{\theta}].$ A similar argument shows that 
$\psi_*([P_{J'',(2p-2)}^{F(\theta)}])=[P_{I'',(2p)}^{\theta}].$ Now 
$$\psi(P_{J'',(2p-3,2p-2)}^{F(\theta)})=\psi\left(\frac{P_{J''}^{F(\theta)}}{2}(1+e(-\frac{{\pf}^{\theta}_{(1,2,2p-1,2p)}}{2\theta_{12}})V_{2p-1}V_{2p}W)\right).$$
Like before, this element is homotopic to 
$$P_{I''}^{\theta}\cdot\frac{e}{2}\left(1+e(-\frac{\theta_{12}\theta_{2p-1 2p}}{2\theta_{12}})U_{2p-1}U_{2p}W\right)=P_{I''}^{\theta}\cdot\frac{e}{2}\left(1+e(-\frac{1}{2}\theta_{2p-1 2p})U_{2p-1}U_{2p}W\right).$$ 
Hence $\psi_*([P_{J'',(2p-3,2p-2)}^{F(\theta)}])=[P_{I'',(2p-1,2p)}^{\theta}].$ Now the above commutative diagram (involving the K-theory) gives the result.
\end{proof}

\subsection*{The computation of $\K_0(A_\theta \rtimes \Z_2)$}
 With the above results in hand, we now come to the computation of $\K_0(A_\theta \rtimes \Z_2).$ Let us first see some lower dimensional cases, i.e. the cases $n=2,3$.

Assume $\theta_{12} \in (\frac{1}{2},1)$ is an irrational number. 	We will compute the K-theory $\K_0(A_{\theta_{12}}\rtimes \Z_2)$ and write down an explicit basis.  Note that $A_{\theta_{12}}\rtimes\Z_2$ is generated by the unitaries $U_1$, $U_2$ and $W$ such that $W^2 = 1, U_1U_2 = e(\theta_{12})U_2U_1, WUW = U^*, WVW = V^*.$ 

We now have
	{\[
\begin{CD}
\K_0 (C(\T)) @>{i_{1*}-i_{2*}}>>\K_0 (C(\T)\rtimes_\alpha  \Z_2)\oplus \K_0 (C(\T)\rtimes \Z_2)@>{j_{1*}+j_{2*}}>> \K_0(A_{\theta_{12}}\rtimes\Z_2)   \\
@A{}AA & &  @VV{e_1}V            \\
 0 @<<{}<  0 @<<{}< \K_1 (C(\T))
\end{CD}
\]} 
 
 In this case $i_{1*}-i_{2*}([1])=([1],-[1]),$ and since we have already described the map $e_1$ (Corollary~\ref{cor:gen_boundary}) we get a basis of $\K_0(A_{\theta_{12}}\rtimes\Z_2)$ which is given by the K-theory classes of the following elements:\\
 \begin{itemize}
 \item $1$;
 \item  $P^{\theta_{12}}_{\emptyset,\emptyset} = \frac{1}{2} (1+W)$;
 \item $P^{\theta_{12}}_{\emptyset,(1)}= \frac{1}{2} (1+U_1W)$;
 \item $P^{\theta_{12}}_{\emptyset,(2)}= \frac{1}{2} (1+U_2W)$;
 \item $P^{\theta_{12}}_{\emptyset,(12)}= \frac{1}{2} (1+e( -\frac{1}{2} {\theta_{12}})U_1U_2W)$;
 \item $P^{\theta_{12}}_{I,\emptyset}= \frac{P^{\theta_{12}}_{I}}{2}(1+W),$ for $I=(1,2).$  
 	\end{itemize}
 Hence a generating set of $\K_0(A_{\theta_{12}}\rtimes \Z_2)$ is given by 
 $\{[1], [P] ~|~ P\in \mathcal{P}_2\},$ using the notations introduced just after Corollary~\ref{cor:gen_boundary}.

Let us now look at the case $n=3.$ Let $\theta=\left(\begin{array}{ccc}0 & \theta_{12} & \theta_{13}\\ -\theta_{12} & 0 & \theta_{23} \\ -\theta_{13} & -\theta_{23} & 0 \end{array}\right)$ be a strongly irrational $3\times 3$ matrix as in Definition~~\ref{a3}. In this case it means that $\theta_{12}, \theta_{13}, \theta_{23} \in (\frac{1}{2},1),$ and they are rationally independent.
	We have the following exact sequence in this case.

{\[
\begin{CD}
\K_0 (A_{\theta_{12}}) @>{i_{1*}-i_{2*}}>>\K_0 (A_{\theta_{12}}\rtimes_\alpha  \Z_2)\oplus \K_0 (A_{\theta_{12}}\rtimes\Z_2)@>{j_{1*}+j_{2*}}>> \K_0 (A_{\theta}\rtimes \Z_2)   \\
@A{}AA & &  @VV{e_1}V            \\
 0 @<<{}<  0 @<<{}< \K_1 (A_{\theta_{12}})
\end{CD}
\]} 
Using the above two-dimensional computations, we first write down a basis for $A_{\theta_{12}}\rtimes_\alpha  \Z_2$ and $A_{\theta_{12}}\rtimes \Z_2.$ A basis of $\K_0(A_{\theta_{12}}\rtimes \Z_2)$ is given by the K-theory classes of the following elements.
 \begin{itemize}
 \item $1$;
 \item  $P^{\theta_{12}}_{\emptyset,\emptyset} = \frac{1}{2} (1+W)$;
 \item $P^{\theta_{12}}_{\emptyset,(1)}= \frac{1}{2} (1+U_1W)$;
 \item $P^{\theta_{12}}_{\emptyset,(2)}= \frac{1}{2} (1+U_2W)$;
 \item $P^{\theta_{12}}_{\emptyset,(12)}= \frac{1}{2} (1+e( -\frac{1}{2} {\theta_{12}})U_1U_2W)$;
 \item $P^{\theta_{12}}_{I,\emptyset}=\frac{P^{\theta_{12}}_{I}}{2}(1+W),$ for $I=(1,2).$  
 	\end{itemize}

Now $A_{\theta_{12}}\rtimes_\alpha \Z_2$ is generated by the unitaries $U_1$, $U_2,$ and $W'=U_3W$ and we have the relations $W'^2 = 1, U_1 U_2 = e(\theta_{12}) U_2 U_1, W'U_1W' = e(\theta_{13})U_1^{-1}, W'U_2W' = e(\theta_{23})U_2^{-1}.$ Use Lemma~\ref{lemma:flip=flip}, and the above two-dimensional computation to get a basis $\{[1], [\tilde P] ~|~ P\in \mathcal{P}_n\}$ of $\K_0(A_{\theta_{12}}\rtimes_\alpha \Z_2)$  by writing $\tilde P$ we indicate that the class of P is taken inside $A_{\theta_{12}}\rtimes_\alpha \Z_2.$  This is explicitly given by the K-theory classes of the following elements. 
 \begin{itemize}
 	
\item $1$;
\item  $\tilde P^{\theta_{12}}_{\emptyset,\emptyset}= \frac{1}{2} (1+U_3W)$;
 \item $\tilde P^{\theta_{12}}_{\emptyset,(1)}= \frac{1}{2} (1+e(-\frac{1}{2}\theta_{13})U_1U_3W)$;
 \item $\tilde P^{\theta_{12}}_{\emptyset,(2)}= \frac{1}{2} (1+e(-\frac{1}{2}\theta_{23})U_2U_3W)$;
 \item $\tilde P^{\theta_{12}}_{\emptyset,(12)}= \frac{1}{2} (1+e( -\frac{1}{2} ({\theta_{12}}+\theta_{13}+\theta_{23}))U_1U_2U_3W)$;
 \item $\tilde P^{\theta_{12}}_{I,\emptyset}=\frac{P^{\theta_{12}}_{I}}{2}(1+U_3W),$ for $I=(1,2).$ 
  \end{itemize}
 Now in the above exact sequence, $i_{1*}-i_{2*}([1])=([1],-[1]),$ and by Proposition~\ref{prop:imageofi_rieffel} (or by Lemma~\ref{lemma:2dim-result}) and we have 
$$i_{2*}([P_{(1,2)}^{\theta_{12}}])=	2[P_{(1,2),\emptyset} ^{\theta_{12}}]-[P_{\emptyset,\emptyset} ^{\theta_{12}}]+[P_{\emptyset,(1)} ^{\theta_{12}}]-[P_{\emptyset, (2)} ^{\theta_{12}}]+[P_{\emptyset, (1,2)} ^{\theta_{12}}].$$
Similarly,
$$i_{1*}([P_{(1,2)}^{\theta_{12}}])=	2[\tilde P_{(1,2),\emptyset} ^{\theta_{12}}]-[\tilde P_{\emptyset,\emptyset} ^{\theta_{12}}]+[\tilde P_{\emptyset,(1)} ^{\theta_{12}}]-[\tilde P_{\emptyset, (2)} ^{\theta_{12}}]+[\tilde P_{\emptyset, (1,2)} ^{\theta_{12}}].$$ 
Since $[1]$ and $[P_{(1,2)}^{\theta_{12}}]$ generate $\K_0(A_{\theta_{12}}),$ 
 looking at the above formulas of $i_{1*}-i_{2*}$ for $[1]$ and $[P_{(1,2)}^{\theta_{12}}],$ we can choose a basis of $\K_0 (A_{\theta_{12}}\rtimes_\alpha  \Z_2)\oplus \K_0 (A_{\theta_{12}}\rtimes \Z_2)$ such that the map $i_{1*}-i_{2*}$ is exactly  $(\id, 0,-\id, 0)^t.$ 
 Now we want to get of basis of $\K_0(A_{\theta}\rtimes \Z_2).$ For this let us consider a basis of $\K_0 (A_{\theta_{12}}\rtimes_\alpha  \Z_2)\oplus \K_0 (A_{\theta_{12}}\rtimes \Z_2)$ given by  $$\{[1]\oplus [0], [0]\oplus [1], [0]\oplus [P],[\tilde P]\oplus [0] ~|~ P\in \mathcal{P}_2\}.$$ We can replace $[1]\oplus [0]$ by $i_{1*}-i_{2*}([1]),$ and $[\tilde P_{\emptyset, (1,2)} ^{\theta_{12}}]\oplus [0]$ by  $i_{1*}-i_{2*}([P_{(1,2)}^{\theta_{12}}])$ to get a new basis of $\K_0 (A_{\theta_{12}}\rtimes_\alpha  \Z_2)\oplus \K_0 (A_{\theta_{12}}\rtimes \Z_2).$ Using this basis in the above sequence, with Corollary~\ref{cor:gen_boundary} and Proposition~\ref{prop:gen_inclusion} in hand, we get the following basis of 
 $\K_0(A_{\theta}\rtimes \Z_2).$   
 
\begin{itemize}
 \item  $1$;
 \item $P^{\theta}_{\emptyset,\emptyset}= \frac{1}{2} (1+W)$;
 \item $P^{\theta}_{\emptyset,(1)}= \frac{1}{2} (1+U_1W)$;
\item  $P^{\theta}_{\emptyset,(2)}= \frac{1}{2} (1+U_2W)$;
\item  $P^{\theta}_{\emptyset,(1,2)}= \frac{1}{2} (1+e( -\frac{1}{2} {\theta_{12}})U_1U_2W)$;
\item  $P_{I,\emptyset}^{\theta} , I=(1,2)$; 
 \item $P^{\theta}_{\emptyset,(3)}= \frac{1}{2} (1+U_3W)$;
 \item $P^{\theta}_{\emptyset,(13)}= \frac{1}{2} (1+e(-\frac{1}{2}\theta_{13})U_1U_3W)$; 
 \item $P^{\theta}_{\emptyset,(23)}= \frac{1}{2} (1+e(-\frac{1}{2}\theta_{23})U_2U_3W)$; 
\item  $P_{I,(3)}^{\theta}=\frac{P^{\theta}_{I}}{2}(1+U_3W), I=(1,2)$; 
 \item $P_{I,\emptyset}^{\theta}=\frac{P^{\theta}_{I}}{2}(1+W), I=(1,3)$;  
 \item $P_{I,\emptyset}^{\theta}=\frac{P^{\theta}_{I}}{2}(1+W), I=(2,3)$. 
  \end{itemize}
   So we have proved that $\K_0(A_{\theta}\rtimes\Z_2)\cong \Z^{12}$ and a basis  of $\K_0(A_{\theta}\rtimes \Z_2)$ may be  given by 
 $\{[1], [P] ~|~ P\in \mathcal{P}_3\}.$ 
  
  We now have our main theorem
  \begin{theorem}\label{thm:main_K_gen}
  Let $\theta$ be a strongly irrational $n\times n$ matrix. Then
  $\K_0(A_{\theta}\rtimes \Z_2)\cong \Z^{3\cdot 2^{n-1}},$ and a generating set  of $\K_0(A_{\theta}\rtimes \Z_2)$ can be  given by 
 $\{[1], [P] ~|~ P\in \mathcal{P}_n\}.$ 
  	\end{theorem}
  
  \begin{proof}
  	We prove the theorem by induction on $n.$  For $n=2,3$ we have already shown the result to be true. Assume that it holds for a number $n.$ Then we must show that the result is true for $n+1.$ The proof is similar to the proof in the case for $n=3$. Let $\theta$ be a strongly irrational $(n+1)\times (n+1)$ matrix. Let $\theta'$ be its upper left $n \times n$ corner. Then, by Natsume's theorem (\cite{Nat85}) we get the following six-term exact sequence.
 
 {\[
\begin{CD}
\K_0 (A_{\theta'}) @>{i_{1*}-i_{2*}}>>\K_0 (A_{\theta'}\rtimes_\alpha  \Z_2)\oplus \K_0 (A_{\theta'}\rtimes\Z_2)@>{j_{1*}+j_{2*}}>> \K_0 (A_{\theta'}\rtimes_\phi \Z_2*\Z_2)   \\
@A{}AA & &  @VV{e_1}V            \\
 0 @<<{}<  0 @<<{}< \K_1 (A_{\theta'})
\end{CD}
\]} 
 By induction hypothesis assume that a generating set  of $\K_0(A_{\theta'}\rtimes\Z_2)$ is  given by 
 $\{[1], [P] ~|~ P\in \mathcal{P}_n\}$ and  for $\K_0(A_{\theta'}\rtimes_\alpha \Z_2),$ the set is  given by 
 $\{[1], [\tilde P] ~|~ P\in \mathcal{P}_n\}$ (as in the 3-dimensional case.)
 
 Now $i_{1*}-i_{2*}([1])=([1],-[1]).$ For each $I\in \mathrm{Minor}(n)\setminus \{\emptyset\},$ we have from Proposition~\ref{prop:imageofi_rieffel}
 
 $$i_{2*}([P_I^{\theta'}])=	2[P_{I,\emptyset} ^{\theta'}]-[P_{I'',\emptyset} ^{\theta'}]+[P_{I'',(i_{2p-1})} ^{\theta'}]-[P_{I'', (i_{2p})} ^{\theta'}]+[P_{I'', (i_{2p-1},i_{2p})} ^{\theta'}],$$ where $I=(i_1,i_2,\dots,i_{2p})$ and $I''$ is obtained from $I$ by deleting the last two numbers.
 Similarly 
  $$i_{2*}([P_I^{\theta'}])=	2[\tilde P_{I,\emptyset} ^{\theta'}]-[\tilde P_{I'',\emptyset} ^{\theta'}]+[\tilde P_{I'',(i_{2p-1})} ^{\theta'}]-[\tilde P_{I'', (i_{2p})} ^{\theta'}]+[\tilde P_{I'', (i_{2p-1},i_{2p})} ^{\theta'}].$$ 
  
  Since $[1]$ and $\{[P_{I}^{\theta'}]$ for  $I\in \mathrm{Minor}(n)\setminus \{\emptyset\}\},$ generate $\K_0(A_{\theta'}),$ 
 looking at the above formulas of $i_{1*}-i_{2*}$ for $[1]$ and $[P_{I}^{\theta'}],$ we can choose a basis of $\K_0 (A_{\theta'}\rtimes_\alpha  \Z_2)\oplus \K_0 (A_{\theta'}\rtimes \Z_2)$ such that the map $i_{1*}-i_{2*}$ is exactly  $(\id, 0,-\id, 0)^t.$ This immediately gives that $\K_0(A_{\theta}\rtimes \Z_2)\cong \Z^{3\cdot 2^{n}}.$  
To find a basis of $\K_0(A_{\theta}\rtimes \Z_2),$ note that a basis of $\K_0 (A_{\theta'}\rtimes_\alpha  \Z_2)\oplus \K_0 (A_{\theta'}\rtimes\Z_2)$ is given by  $$\{[1]\oplus [0], [0]\oplus [1], [0]\oplus [P],[\tilde P]\oplus [0] ~|~ P\in \mathcal{P}_n\}.$$ But as in the 3-dimensional case, we can replace $[1]\oplus [0]$ by $i_{1*}-i_{2*}([1]),$ and $[\tilde P_{I'', (i_{2p-1},i_{2p-2})} ^{\theta'}]\oplus [0]$ by  $i_{1*}-i_{2*}([P_{I}^{\theta'}]),$ for each $I\in \mathrm{Minor}(n)\setminus \{\emptyset\},$ to get a new basis of $\K_0 (A_{\theta'}\rtimes_\alpha  \Z_2)\oplus \K_0 (A_{\theta'}\rtimes \Z_2).$ Using this basis in the above exact sequence along with Corollary~\ref{cor:gen_boundary} and Proposition~\ref{prop:gen_inclusion}, we get our desired basis of 
 $\K_0(A_{\theta}\rtimes \Z_2).$ 
 \end{proof}
  
  We immediately get the following result, which was also obtained in \cite{ELPW10} (see the proof of Theorem 6.6 there).   
  
  \begin{corollary}\label{thm:main_kthoeryall} For any $\theta \in \mathcal{T}_n,$ 
	$$
\K_{0}(A_{\theta}\rtimes \Z_2) \cong \mathbb{Z}^{3 \cdot 2^{n-1}}, \quad \K_{1}(A_{\theta}\rtimes \Z_2) = 0 .
$$
 \end{corollary} 
\begin{proof} 

	From the proof of Theorem 6.6 in \cite{ELPW10}, it is enough  prove the statement for one  $\theta \in \mathcal{T}_n.$ Since Corollary~\ref{cor:class_of_strongly} of Appendix~\ref{sec:strongly_irr} gives a large class of examples of strongly totally irrational matrices, use Theorem~\ref{thm:main_K_gen}.
\end{proof}

  \section{Generators of  $\K_0(A_\theta\rtimes \Z_2)$ for a general $\theta$}\label{sec:gen_general}

 Using ideas from \cite{ELPW10}, in this section we construct a continuous field of projective modules,  which will play the major role to compute the generators of $\K_0(A_\theta\rtimes \Z_2)$ for a general $\theta.$ A similar idea was used in \cite{Cha20} for $A_\theta$ to compute the generators of $\K_0(A_\theta)$ for a general $\theta,$ however there was a gap in the arguments therein. We fix the arguments and show that indeed such a construction of a field is possible for $A_\theta.$ We then extend these ideas to $A_\theta\rtimes \Z_2.$

 Let $G$ be a discrete group and let $\Omega$ be a $C([0,1],\T)$-valued 2-cocycle on $G$ (as in \cite[Section 1]{ELPW10}). One can then define the reduced twisted crossed product $C^*$-algebra $C([0,1])\rtimes_{ \Omega} G$ just like the twisted group $C^*$-algebras, where $G$ acts trivially on $C([0,1]).$ Here the underlying convolution algebra is the algebra of $\ell^1$ functions on $G$ with values in $C([0,1])$. Given any $t\in [0,1]$, the function 
$$
\omega_t := \Omega(\cdot, \cdot)(t)
$$
is a $\T$-valued 2-cocycle on $G$. There is a canonical map (called the \emph{evaluation map})
$$
\mathrm{ev}_t: C([0,1])\rtimes_{ \Omega} G\to C^*(G, \omega_t)
$$
such that for each function $f\in \ell^1(G, C([0,1]))$ and each $x\in G$ we have $(\mathrm{ev}_t(f))(x) = (f(x))(t)$.

\begin{theorem}
\cite[Corollary 1.11]{ELPW10} \label{thm:ev}
	If $G$ satisfies the Baum--Connes conjecture with coefficients, then the evaluation map $\ev_t$ induces an isomorphism on K-theory.
	\end{theorem}

	Let $\Omega$ be a $C([0,1],\T)$-valued 2-cocycle on $\Z^n$ such that
	$$
	\Omega( \cdot, \cdot   )(t) = \omega_{\theta_t},~ \theta_t\in \mathcal{T}_n
	$$
	is a 2-cocycle on $\Z^n$ as in Example~\ref{ex:nt} for all $t.$ Also let $\widetilde{\Omega}$ be the $C([0,1],\T)$-valued 2-cocycle on $\Z^n\rtimes \Z_2$ so that $\widetilde{\Omega}( \cdot, \cdot   )(t)=\omega'_{\theta_t}$ (as in Example~\ref{ex:orbifold}). Since the groups $\Z^n$ and $\Z^n\rtimes \Z_2$ satisfy the Baum--Connes conjecture with coefficients (see \cite{HK01}), by Theorem~\ref{thm:ev} both evaluation maps
	$$
	\mathrm{ev}_t: C([0,1])\rtimes_{ \Omega} \Z^n\to C^*(\Z^n, \omega_{\theta_t}),
	$$
	$$
	\mathrm{ev}_t: C([0,1])\rtimes_{ \widetilde{\Omega}} (\Z^n \rtimes \Z_2) \to C^*(\Z^n \rtimes \Z_2, \omega'_{\theta_t})
	$$
	induce isomorphisms at the level of $\K_0$ and $\K_1$. As in the case of twisted group $C^*$-algebras, there is an identification
	$$
	C([0,1])\rtimes_{ \widetilde{\Omega}} (\Z^n\rtimes \Z_2) \xrightarrow{\cong} (C([0,1])\rtimes_{ \Omega} \Z^n)\rtimes \Z_2
	$$
	that respects the evaluation maps (see \cite[Remark 2.3]{ELPW10}).

As a result of the above discussions, using Theorem~\ref{thm:ev} we have (see also \cite[Remark 2.3]{ELPW10}) the following theorems.   
\begin{theorem}\label{thm:elpw_gen_nt}
Let $[p_1],[p_2],\cdots , [p_m] \in \K_0(C([0,1])\rtimes_{ \Omega} \Z^n)$. 
	Then the following are equivalent:
	
	\begin{enumerate}
		\item $[p_1],[p_2],\cdots , [p_m]$ form a basis of $\K_0(C([0,1])\rtimes_{ \Omega} \Z^n )$.  
		\item For some $t \in [0,1]$, the evaluated classes $[\ev_t(p_1)],[\ev_t(p_2)], \cdots ,[\ev_t(p_m)]$ form a basis of $\K_0(C^*(\Z^n,\omega_{\theta_t}))$. 
	  \item For every $t \in [0,1]$, the evaluated classes $[\ev_t(p_1)],[\ev_t(p_2)], \cdots , [\ev_t(p_m)]$ form a basis of $\K_0(C^*(\Z^n,\omega_{\theta_t}))$. 
\end{enumerate}
\end{theorem}

\begin{theorem}\label{thm:elpw_gen}
\label{elpwmain}Let $[p_1],[p_2],\cdots , [p_m] \in \K_0((C([0,1])\rtimes_{ \Omega} \Z^n    )\rtimes \Z_2)$. 
	Then the following are equivalent:
	
	\begin{enumerate}
		\item $[p_1],[p_2],\cdots , [p_m]$ form a basis of $\K_0((C([0,1])\rtimes_{ \Omega} \Z^n    )\rtimes \Z_2)$.  
		\item For some $t \in [0,1]$, the evaluated classes $[\ev_t(p_1)],[\ev_t(p_2)], \cdots ,[\ev_t(p_m)]$ form a basis of $\K_0(C^*(\Z^n,\omega_{\theta_t})\rtimes \Z_2)$. 
	  \item For every $t \in [0,1]$, the evaluated classes $[\ev_t(p_1)],[\ev_t(p_2)], \cdots , [\ev_t(p_m)]$ form a basis of $\K_0(C^*(\Z^n,\omega_{\theta_t})\rtimes \Z_2)$. 
\end{enumerate}
\end{theorem}

 Since $\mathbb{Z}^{n} \rtimes \Z_2$ is a discrete group, there is a canonical map from $\mathbb{Z}^{n} \rtimes \Z_2$ into the group of unitaries of $C([0,1])\rtimes_{ \widetilde{\Omega}} (\Z^n \rtimes \Z_2)$. Let $\mathcal{U}_{i} \in C([0,1])\rtimes_{ \widetilde{\Omega}} (\Z^n \rtimes \Z_2)$ be the images of the generators $x_i \in \mathbb{Z}^{n}\subset \mathbb{Z}^{n} \rtimes F$ (here $x_i$ is as in Example~\ref{ex:nt}) and $\mathcal{W}$ be the image of the generator of $\Z_2 \subset \mathbb{Z}^{n} \rtimes \Z_2$ under the canonical map. Also, we denote by $\mathcal{U}_{i}$ the image of $x_i \in \mathbb{Z}^{n}$ in $C([0,1])\rtimes_{ \Omega} \Z^n.$

 Next we will construct our required $\Omega.$ Consider the matrix $Z \in \mathcal{T}_n$ whose entries above the diagonal are all 1: 
\begin{equation}\label{eq:def_of_Z}
	Z = 
   \left( \begin{array}{cccccccc}
0 & 1 & \cdots  &  &  & \cdots &1  \\
-1 & \ddots &\ddots  & &   &  &\vdots \\
\vdots & \ddots  &   &  & &  &\\
 &  & &  &  &  &\\
&  &   & &  & \ddots  & \vdots\\
  \vdots &  &  & &   \ddots  & \ddots & 1\\
-1 & \cdots &  &  &  \cdots &  -1 & 0\\
\end{array} \right) .
\end{equation}
If $\theta \in \mathcal{T}_n$, then by a translate of $\theta$ we understand any element $\theta^{\text {tr }} \in \mathcal{T}_n$ such that $\theta-\theta^{\text {tr }} \in M_n(\mathbb{Z})$. We then have $A_\theta=A_{\theta^{\text {tr }}}$ since both matrices induce the same commutation relation on the generating unitaries (or, alternatively, since the corresponding circle-valued 2-cocycles $\omega_\theta$ and $\omega_{\theta^{\text {tr }}}$ coincide). In particular, this holds for $\theta^{\text {tr }}=\theta+n Z$ for every $n \in \mathbb{Z}$. Now, for any $\theta \in \mathcal{T}_n$ there exists some $n \in \mathbb{Z}$ such that all pfaffian minors of $\theta^{\text {tr }}=\theta+n Z$ are positive (see \cite[Proposition 4.6]{Cha20}). Thus, replacing $\theta$ by $\theta^{\text {tr }}$ if necessary, we may assume without loss of generality that all the pfaffian minors of $\theta$ are positive.

Let us also fix a strongly totally irrational matrix $\psi \in \mathcal{T}_n$ as in Appendix~~\ref{sec:strongly_irr}. Then, by definition (see Definition~\ref{a3}), all pfaffian minors of $\psi$ are positive as well. Passing, again, to suitable translates $\theta^{\text {tr }}$ and $\psi^{\text {tr }}$ of $\theta$ and $\psi$, if necessary, we may apply Proposition~\ref{prp:translate} to obtain continuous paths $\left[0, \frac{1}{2}\right] \ni t \mapsto \theta(t):=(1-2 t) \theta+2 t Z \in \mathcal{T}_n$ and $\left[\frac{1}{2}, 1\right] \ni t \mapsto \psi(t):=(2 t-$ 1) $\psi+(2-2 t) Z \in \mathcal{T}_n$ such that all pfaffian minors of $\theta(t)$ and $\psi(t), t \in[0,1]$, are positive. Gluing both paths at $\theta\left(\frac{1}{2}\right)=Z=\psi\left(\frac{1}{2}\right)$ we obtain a continuous path $[0,1] \ni t \rightarrow \nu(t) \in \mathcal{T}_n$ connecting $\theta$ with $\psi$ such that all pfaffian minors of $\nu(t), t \in[0,1]$, are positive. As a consequence, we can now formulate

\begin{proposition}\label{prp:path_positive_pfaffian}
	For each matrix $\theta \in \mathcal{T}_n$ there exists a strongly totally irrational matrix $\psi \in \mathcal{T}_n$ and a continuous path $[0,1] \ni t \mapsto \nu(t) \in \mathcal{T}_n$ with the following properties:
	\begin{enumerate}[label=(\alph*)]
		\item All pfaffian minors of $\nu(t), t \in[0,1]$, are positive.
		\item The endpoints $\nu(0)=\theta^{\text {tr }}$ and $\nu(1)=\psi^{\text {tr }}$ are translates of $\theta$ and $\psi$, respectively.
	\end{enumerate}
\end{proposition}
\noindent The path in the above proposition now determines a $C([0,1], \mathbb{T})$-valued 2-cocycles $\Omega_\nu$ on $\mathbb{Z}^n,$ and $\widetilde{\Omega}_\nu$ on $\mathbb{Z}^n\rtimes \Z_2$ by defining
$$
\Omega_\nu(\cdot, \cdot)(t)=\omega_{\nu(t)}(\cdot, \cdot)\quad \text{and}\quad \widetilde{\Omega}_{\nu}(\cdot, \cdot)(t)=\omega'_{\nu(t)}(\cdot, \cdot), \quad t \in[0,1].
$$


	 In the following we will construct a class of projective modules over $(C([0,1])\rtimes_{ \Omega_{\nu}} \Z^n)\rtimes \Z_2$ such that the class restricted to each point $t$ of $[0,1]$ gives a basis of $\K_0(C^*(\Z^n,\omega_{{\nu}(t)})\rtimes \Z_2).$

  Now using \cite[Theorem 3.3]{Cha20}, for each $I\in \mathrm{Minor}(n)\setminus\{\emptyset\},$ we can construct a projective module $\mathcal{E}_{I}^{[0,1]}$ over $C([0,1])\rtimes_{\Omega_{\nu}} \Z^n$ such that $\mathcal{E}_{I}^{[0,1]}$ restricted to  $t\in [0,1]$ is the module $\mathcal{E}_{I}^{{\nu}(t)}.$ The idea of such a construction is as follows: using the notations of Section~\ref{sec:rie}, for $I\in \mathrm{Minor}(n)\setminus\{\emptyset\}, |I|=2p,$ consider the algebra $C([0,1])\rtimes_{\rho(R^{\Sigma}_I)\Omega_{\nu}} \Z^n,$ $\rho(R^{\Sigma}_I)\Omega_{\nu}(\cdot ,\cdot)(t):=\omega_{\rho(R^{\Sigma}_I){\nu}(t)},$ where $\rho(R^{\Sigma}_I)$ is as in the discussion after Remark~\ref{rmk:trace_of_heisenberg}. Note that as in before, $C([0,1])\rtimes_{\rho(R^{\Sigma}_I)\Omega_{\nu}} \Z^n$ is also canonically isomorphic to $C([0,1])\rtimes_{\Omega_{\nu}} \Z^n.$ Also define  $C([0,1])\rtimes_{g_{I,\Sigma}(\Omega_{\nu})} \Z^n,$ $g_{I,\Sigma}(\Omega_{\nu})(\cdot,\cdot)(t):=\omega_{g_{I,\Sigma}{\nu}(t)}.$ For $M= \R^p \times \Z^{n-2p},$ define the space $\mathscr{S}(M, [0,1])$ consisting of all complex functions on $M \times [0,1]$ which are smooth and rapidly decreasing in the first  variable and continuous in the second variable in each derivative of the first variable. Then using the equations \ref{eq:module1},~\ref{eq:module2},~\ref{eq:module3},~\ref{eq:module4} fibre-wise, a suitable of completion of $\mathscr{S}(M, [0,1]),$ $\mathcal{E}_{I}^{[0,1]},$ becomes a $C([0,1])\rtimes_{g_{I,\Sigma}(\Omega_{\nu})} \Z^n-C([0,1])\rtimes_{\Omega_{\nu}} \Z^n$ strong Morita equivalence bi-module. Here in this construction we use the fact that $\pf(({\nu}(t))_I)$ is positive. Since $C([0,1])\rtimes_{\Omega_{\nu}} \Z^n$ is unital, $\mathcal{E}_{I}^{[0,1]}$ is also a projective module over $C([0,1])\rtimes_{\Omega_{\nu}} \Z^n.$ The detailed proof of this construction can easily be deduced from the proof \cite[Theorem 3.3]{Cha20} (see also the remark below).
  
  \begin{remark}
  	In \cite{Cha20}, in the construction of $C([0,1])\rtimes_{\Omega_{\nu}} \Z^n$ ($C^{*}\left(\mathbb{Z}^{n} \times I, \Omega\right)$ in the notations of \cite{Cha20}) $\Omega_{\nu}$ was dependent on $I,$ therefore the proof of \cite[Theorem 4.7]{Cha20} was incomplete. But we fix the proof in this paper by making $\Omega_{\nu}$ independent of $I$ and this results in Theorem~\ref{thm:main_K_gen_proj_nt}.   
   \end{remark}

  Using Proposition~\ref{prop:mainprop}, with the fibre-wise flip actions as in Equation~\ref{eq:flip_action_module}, it is easily checked that $\mathcal{E}_{I}^{[0,1]}$ becomes a projective module over the crossed product $(C([0,1])\rtimes_{ \Omega_{\nu}} \Z^n)\rtimes \Z_2.$      
  
We now construct a generating set of $\K_0((C([0,1])\rtimes_{ \Omega_{\nu}} \Z^n)\rtimes \Z_2)$ using the projective modules over $(C([0,1])\rtimes_{ \Omega_{\nu}} \Z^n)\rtimes \Z_2$ constructed above. Using the notations of Section~\ref{sec:k-theory}, for each $I\in \mathrm{Minor}(n),$ fix a $J\subseteq I^c.$ For $J=(j_1,j_2,\dots,j_q)$, define $\mathcal{U}_J:=\mathcal{U}_{j_1}\mathcal{U}_{j_2}\cdots \mathcal{U}_{j_q}\in C([0,1])\rtimes_{\Omega_{\nu}} \Z^n$, and for $J=\emptyset,$ $\mathcal{U}_J:=1\in C([0,1])\rtimes_{\Omega_{\nu}} \Z^n.$ For a real function $f \in C([0,1],\R),$ define $\mathrm{exp}(f)\in C([0,1],\T)$ by $\mathrm{exp}(f)(t)=e(f(t)).$ Let $r_J^{[0,1]}\in C([0,1],\R)$ be the real function such that $\mathcal{U}_J=\mathrm{exp}(-2r^{[0,1]}_J)\mathcal{U}_{\bar{J}}.$  Now $\mathrm{exp}(r_J)\mathcal{U}_J\mathcal{W}=:\mathcal{W}_J$ is a self-adjoint unitary in $(C([0,1])\rtimes_{\Omega_{\nu}} \Z^n)\rtimes \Z_2.$ One quickly checks that  $\mathcal{W}_J\mathcal{U}_{i}\mathcal{W}_J=\mathrm{exp}(r_{J,i})\mathcal{U}_{i}^{-1},$ for some function $r_{J,i}\in C([0,1],\R).$ A modified version of Lemma~\ref{lemma:flip=flip} shows that if we set $\widetilde{\mathcal{U}_{i}}=\mathrm{exp}(-\frac{r_{J,i}}{2})\mathcal{U}_{i}$ (so that we have  $\mathcal{W}_J\widetilde{\mathcal{U}_{i}}\mathcal{W}_J=\widetilde{\mathcal{U}_{i}}^{-1}$), we have an isomorphism from $(C([0,1])\rtimes_{\Omega_{\nu}} \Z^n)\rtimes \Z_2$ to $(C([0,1])\rtimes_{\Omega_{\nu}} \Z^n)\rtimes \Z_2$ sending $\mathcal{U}_i$ to $\widetilde{\mathcal{U}_{i}}$ and $\mathcal{W}$ to $\mathcal{W}_J.$ Here in the last crossed product $(C([0,1])\rtimes_{\Omega_{\nu}}) \Z^n\rtimes \Z_2,$ the canonical unitary of $\Z_2$ is identified with $\mathcal{W}_J.$ Using this identification, for each $I\in \mathrm{Minor}(n)\setminus\{\emptyset\},$ and $J\subseteq I^c.$  $\mathcal{E}_{I}^{[0,1]}$ becomes a module over $(C([0,1])\rtimes_{\Omega_{\nu}}) \Z^n\rtimes \Z_2.$ We call this module $\mathcal{E}_{I,J}^{[0,1]}.$ For $I=\emptyset,$ and $J\subseteq I^c,$  $\mathcal{E}_{I,J}^{[0,1]}:=\frac{1}{2}(1+\mathcal{W}_J).$ 
	  We consider the following family of elements in $\K_0((C([0,1])\rtimes_{\Omega_{\nu}}) \Z^n\rtimes \Z_2)$ 
	  
	 $$ \underset{I\in \mathrm{Minor}(n)}{\cup}\left\{\mathcal{E}_{I,J}^{[0,1]}~\suchthat ~J\subseteq I^c,|J|\le 2 \right\}.$$
	  
	  Our next step is to show that if we consider the K-theory classes $[\mathcal{E}_{I,J}^{[0,1]}]$ in the above family inside $(C([0,1])\rtimes_{\Omega_{\nu}}) \Z^n\rtimes \Z_2,$ then the K-theory classes $[\ev_1(\mathcal{E}_{I,J}^{[0,1]})],$ along with $[1]$ provide a basis of $\K_0(A_{\nu(1)}\rtimes\Z_2)=\K_0(A_{\psi^\mathrm{tr}}\rtimes\Z_2)=\K_0(A_{\psi}\rtimes\Z_2).$  Let us write $\mathcal{E}_{I,J}^{\psi^\mathrm{tr}}$ for $\ev_1(\mathcal{E}_{I,J}^{[0,1]}).$ Since $\psi$ is totally irrational, the translate of $\psi,$ $\psi^\mathrm{tr},$ is also totally irrational. From the pfaffian summation formula (as in Equation~\ref{eq:sumofpaff}) if we compute the pfaffian of $\psi^\mathrm{tr}_I,$ for a fixed $I\in \mathrm{Min}(n),$ we see that this is exactly $\Tr(P_I^\psi)+\underset{|I'|<|I|}{\sum} c_{I'}\Tr(P_{I'}^\psi),$ $c_{I'}\in \Z.$ This also coincides with the trace of $[\mathcal{E}_{I}^{\psi^\mathrm{tr}}]=[\ev_1(\mathcal{E}_{I}^{[0,1]})]$ considering $\mathcal{E}_{I}^{[0,1]}$ as a projective module over $C([0,1])\rtimes_{\Omega_{\nu}} \Z^n.$  Hence $[P_I^\psi]+\underset{|I'|<|I|}{\sum} c_{I'}[P_{I'}^\psi]= [\mathcal{E}_{I}^{\psi^\mathrm{tr}}]$ as the trace map is injective for $A_{\psi^\mathrm{tr}}.$ Now for $J\subseteq I^c,$ $[\mathcal{E}_{I,J}^{\psi^\mathrm{tr}}]$ and $[P_{I,J}^\psi]+\underset{|I'|<|I|}{\sum} c_{I'}[P_{I',J}^\psi]$ coincide inside $\K_0(A_{\psi^\mathrm{tr}}\rtimes\Z_2)$ as they are extended from the same element in $\K_0(A_{\psi^\mathrm{tr}})$ using the same method i.e. the Green--Julg map. Now using Theorem~\ref{thm:main_K_gen}, we know that the elements of the set $\underset{I\in \mathrm{Minor}(n)}{\cup}\left\{[P_{I,J}^\psi]+\sum_{|I'|<|I|}c_{I'}[P_{I',J}^\psi]~\suchthat ~J\subseteq I^c,|J|\le 2 \right\}$  along with $[1]$ form a basis of $\K_0(A_{\psi}\rtimes\Z_2)=\K_0(A_{\psi^\mathrm{tr}}\rtimes\Z_2).$  Considering $[\mathcal{E}_{I,J}^{[0,1]}]$ inside $(C([0,1])\rtimes_{\Omega_{\nu}}) \Z^n\rtimes \Z_2,$ let us denote $\ev_0(\mathcal{E}_{I,J}^{[0,1]})$ by $\mathcal{E}_{I,J}^{\theta^\mathrm{tr}}$. Let 
	  \begin{equation}\label{eq:description_of_general}
	  	\operatorname{Proj}_n=\left\{\mathcal{E}_{I, J}^{\theta^{\text {tr }}}: I \in \operatorname{Minor}(n), J \subseteq I^c,|J| \leq 2\right\}.
\end{equation}

	    Now if we apply Theorem~\ref{thm:elpw_gen} to the cocycle ${\Omega}_\nu$ constructed just after Proposition~\ref{prp:path_positive_pfaffian}, as a result of the above discussion we immediately get the following theorem. 
	  \begin{theorem}\label{thm:main_K_gen_proj} Let $\theta \in \mathcal{T}_n$. Then there exists a suitable translate $\theta^{\text {tr }} \in \mathcal{T}_n$ for $\theta$ such that the set of classes $\left\{[1],[\mathcal{E}] \mid \mathcal{E} \in \operatorname{Proj}_n\right\},$ where
$\operatorname{Proj}_n$ is as in Equation~\ref{eq:description_of_general}, generate $\K_0(A_{\theta}\rtimes \Z_2).$
	 \end{theorem}
	 
Using a similar idea as in the proof of the above theorem along with Theorem~\ref{thm:elpw_gen_nt} we also get the following theorem.
	 
	  \begin{theorem}\label{thm:main_K_gen_proj_nt}(cf. \cite[Theorem 4.7]{Cha20}) Let $\theta \in \mathcal{T}_n.$ Then there exists a suitable translate $\theta^{\text {tr }} \in \mathcal{T}_n$ for $\theta$ such that the set of classes 
$$\underset{I\in \mathrm{Minor}(n)\setminus\{\emptyset\}}{\cup}\left\{[1], ~[\mathcal{E}_{I}^{\theta^\mathrm{tr}}] \right\}$$ generate $\K_0(A_{\theta}).$
	 \end{theorem}
	 
	 From Equation~\ref{eq:trace_of_extended_module} it is clear that 
$\Tr^{\Z_2}_{\theta^\mathrm{tr}}([\mathcal{E}_{I,J}^{\theta^\mathrm{tr}}])=\frac{\pf(\theta^\mathrm{tr}_I)}{2}.$ It follows that \begin{equation}\label{eq:trace_of_flipped}
\Tr^{\Z_2}_{\theta^\mathrm{tr}}(\K_0(A_{\theta^\mathrm{tr}}\rtimes \Z_2))=\frac{\Tr_{\theta^\mathrm{tr}}(\K_0(A_{\theta^\mathrm{tr}}))}{2},	
\end{equation} which in turn gives that 
$\Tr^{\Z_2}_{\theta}(\K_0(A_{\theta}\rtimes \Z_2))=\frac{\Tr_{\theta}(\K_0(A_{\theta}))}{2},$ from the computation of $\pf(\theta^\mathrm{tr}_I)$ using the pfaffian summation formula (Equation~\ref{eq:sumofpaff}). This result was already obtained in \cite[Example 4.5]{Cha21}.

 \section{isomorphism classes of $A_\theta\rtimes \Z_2$}\label{sec:iso_classes}
 
 In this section we give an application of the explicit K-theory computations from the previous section. We start with the following definition.

\begin{definition}\label{def:nondegenerate}
\normalfont
	A skew symmetric real $n \times n$ matrix $\theta$ is called \emph{non-degenerate} if whenever $x \in \mathbb{Z}^{n}$ satisfies $e(\langle x, \theta y\rangle)=1$ for all $y \in \Z^n,$ then $x=0 .$ 
\end{definition}

The theorem which we want to prove in this section is the following. 

\begin{theorem}\label{thm:main_iso}
		Let $\theta_{1}, \theta_{2} \in \mathcal{T}_n $ be non-degenerate.
		Let $\Z_2$ act on $A_{\theta_{1}}$ and $A_{\theta_{2}}$ by the flip actions.  Then $A_{\theta_{1}}\rtimes \Z_2$ is isomorphic to $A_{\theta_{2}}\rtimes \Z_2$ if $A_{\theta_{1}}$ is isomorphic to $A_{\theta_{2}}.$ Moreover, if any one of $\theta_{1}, \theta_{2}$ is totally irrational, the converse is true.
 	 \end{theorem}

 We need some preparation before proving the above theorem. 

 	  \begin{proposition}\label{prop:phi06}(\cite[Proposition 3.7]{Phi06})
	 Let $A$ be a simple infinite dimensional separable unital nuclear C*-algebra with tracial rank zero and which satisfies the Universal Coefficient Theorem. Then $A$ is a simple $A H$ algebra with real rank zero and no dimension growth. If $\K_{*}(A)$ is torsion free, $A$ is an AT algebra. If, in addition, $\K_{1}(A)=0$, then $A$ is an $A F$ algebra.
\end{proposition}

Let $\theta \in \mathcal{T}_{n}$ be non-degenerate. Then the following are known.

\begin{itemize}
	\item $A_{\theta}$ is a simple C*-algebra (even the converse is true: simplicity of $A_\theta$ implies $\theta$ must be non-degenerate) with a unique tracial state (\cite[Theorem 1.9]{Phi06});
	\item $A_{\theta}$ is tracially AF (\cite[Theorem 3.6]{Phi06});
	\item If $\beta$ is an action of a finite group on $A_{\theta}$  which has the tracial Rokhlin property (see \cite[Section 5]{ELPW10}), $A_{\theta} \rtimes_{\beta} F$ is a simple C*-algebra with tracial rank zero (\cite[Corollary 1.6, Theorem 2.6]{Phi11}). Also, $A_{\theta} \rtimes_{\beta} F$ has a unique tracial state (\cite[Proposition 5.7]{ELPW10});
	\item The flip action of $\Z_2$ on $A_\theta$ has the tracial Rokhlin property (\cite[Lemma 5.10 and Theorem 5.5]{ELPW10});
	\item $A_{\theta} \rtimes \Z_2$ satisfies the Universal Coefficient Theorem (see the proof of Theorem 6.6 in \cite{ELPW10}).
	\end{itemize}

With the above list of results along with Proposition~\ref{prop:phi06} and Theorem~\ref{thm:main_kthoeryall}, we readily have the following corollary. 

\begin{corollary}\label{cor:flip_AF}
	Let $\theta \in \mathcal{T}_{n}$ be non-degenerate. Then  $A_\theta\rtimes \Z_2$ is an $A F$ algebra.
\end{corollary}
\noindent The above corollary was first obtained in \cite[Theorem 6.6]{ELPW10}. We are now ready to prove Theorem~\ref{thm:main_iso}. 

\begin{proof}[Proof of Theorem~\ref{thm:main_iso}]

Since taking a translate of $\theta_i$, $i=1,2$ does not change the isomorphism classes of $A_{\theta_i}$ or $A_{\theta_i}\rtimes \Z_2,$ for this proof we may replace the $\theta_i$ by any of their translates. 

To prove the last part, WLOG assume that $\theta_2$ is totally irrational. Let there be a $*$-isomorphism $f$ from $A_{\theta_{1}}\rtimes \Z_2$ to $A_{\theta_{2}}\rtimes \Z_2.$ Then  $\Tr^{\Z_2}_{\theta_{1}}$ and $\Tr^{\Z_2}_{\theta_{2}}$ have the same range, and hence using Equation~~\ref{eq:trace_of_flipped} $\Tr_{\theta_{1}}$ and $\Tr_{\theta_{2}}$ also have the same range. Now to show that $A_{\theta_{1}}$ and $A_{\theta_{2}}$ are $*$-isomorphic, it is enough to find an isomorphism $g: \K_{0}\left(A_{\theta_{1}}\right) \rightarrow \K_{0}\left(A_{\theta_{2}}\right)$ such that $\Tr_{\theta_{2}} \circ g=\Tr_{\theta_{1}}$ and $g([1])=[1].$ Indeed, $g$ is then an order isomorphism by \cite[Proposition 3.7]{BCHL18}, and using classification of tracially AF algebras by Lin (\cite[Theorem 5.2]{Lin04}), we conclude that $A_{\theta_{1}}$ and $A_{\theta_{2}}$ are $*$-isomorphic.
	Let us now see the existence of the isomorphism $g$. Denote the ranges of $\Tr_{\theta_{1}}$ and $\Tr_{\theta_{2}}$ by $R_1$ and $R_2,$ respectively. Since $R_1$ and  $R_2$ are finitely generated subgroups of $\R,$ they are free. Also $R_1 =  R_2$ implies that they have the same rank. Now we have the following exact sequences :

\begin{center}

	\begin{tikzcd}
  0 \arrow[r] & \operatorname{ker}\left(\Tr_{\theta_{1}}\right) \arrow{r} & \K_{0}(A_{\theta_{1}})  \arrow[r, "\Tr_{\theta_{1}}"]& R_1           \arrow{r} & 0 
    \end{tikzcd}
    
    	\end{center}
    	\begin{center}

	\begin{tikzcd}
  0 \arrow[r] & \operatorname{ker}\left(\Tr_{\theta_{2}}\right) \arrow{r} & \K_{0}(A_{\theta_{2}})  \arrow[r, "\Tr_{\theta_{2}}"']& R_2           \arrow{r} & 0 
    \end{tikzcd}
    
    	\end{center}
Note that the above sequences split since the K-groups are torsion free. Now $\operatorname{ker}\left(\Tr_{\theta_{1}}\right)$ and $\operatorname{ker}\left(\Tr_{\theta_{2}}\right)$ are finitely generated abelian groups of the same rank. So there exists   an isomorphism $\psi$ between them. Now $g$ is defined as $\psi \oplus \phi,$ where $\phi$ is the map between $R_1$ and $R_2$ given by multiplication with $1.$ Clearly $\Tr_{\theta_{2}} \circ g=\Tr_{\theta_{1}},$ since the following diagram commutes. 

	\begin{center}
\begin{tikzcd}
\K_{0}(A_{\theta_{1}}) \arrow[r, "\Tr_{\theta_{1}}"] \arrow[d, "g"]
& R_1 \arrow[d, "\phi"] \\
\K_{0}(A_{\theta_{2}}) \arrow[r, "\Tr_{\theta_{2}}"' ]
& R_2
\end{tikzcd}
	\end{center}
Now $g([1])=[1]$ follows from $\Tr_{\theta_{2}} \circ g=\Tr_{\theta_{1}},$ and total irrationality of $\theta_2.$ Indeed, using Theorem~\ref{thm:main_K_gen_proj_nt}, let us write $g([1])$ as a linear combination of the basis elements. As we have $\Tr_{\theta_{2}}( g([1]))=1,$ then the total irrationality of $\theta_{2}$ forces $g([1])$ to be $[1].$

Now assume there is a $*$-isomorphism $f$ from $A_{\theta_{1}}$ to $A_{\theta_{2}}.$ This gives $\Tr_{\theta_{2}} \circ f=\Tr_{\theta_{1}},$ since $A_{\theta_{1}}$ has a unique trace and $f(1)=1.$ So we have an order isomorphism, which we call by  $f$ again, from $\K_0(A_{\theta_{1}})$ to $\K_0(A_{\theta_{2}}).$ Choosing the basis $\{[1],~[\mathcal{E}^{\theta_1}_{I'}]~|~I' \in \mathrm{Minor}(n)\setminus\emptyset\}$  of $\K_0(A_{\theta_1})$ using Theorem~\ref{thm:main_K_gen_proj_nt}, we have 
$$f([\mathcal{E}^{\theta_1}_{I'}])= \sum_{I''\in \mathrm{Minor}(n)} C^{I'}_{I''}[\mathcal{E}^{\theta_2}_{I''}], \quad \text{for}~ I'\in \mathrm{Minor}(n),$$ where $C^{I'}_{I''}\in \Z.$
We have that $A_{\theta_{1}}\rtimes \Z_2$ and $A_{\theta_{2}}\rtimes \Z_2$ are both $AF$ algebras. Now to show that $A_{\theta_{1}}\rtimes \Z_2$ and $A_{\theta_{2}}\rtimes \Z_2$ are isomorphic, it is enough (just like before) to find an isomorphism $f': \K_{0}\left(A_{\theta_{1}}\rtimes \Z_2\right) \rightarrow \K_{0}\left(A_{\theta_{2}}\rtimes \Z_2\right)$ such that $\Tr^{\Z_2
}_{\theta_{2}} \circ f'=\Tr^{\Z_2
}_{\theta_{1}}$ and $f'([1])=[1].$ Now it is enough to define the map  on a set of generators of $\K_0(A_{\theta_{1}}\rtimes \Z_2)$ which is given by $\{[1], [\mathcal{E}],~ \mathcal{E} \in \mathrm{Proj}_n\},$ where 
$$ \mathrm{Proj}_n= \underset{I'\in \mathrm{Minor}(n)}{\cup}\left\{\mathcal{E}_{I',J}^{\theta_1}~\suchthat ~J\subseteq (I')^c,|J|\le 2 \right\}$$ using Theorem~\ref{thm:main_K_gen_proj}.
We define the map $f'$ as follows. $f'([1])=[1],$ 
for $I'\neq \emptyset$
 $$f'([\mathcal{E}_{I',J}^{\theta_1}]) = \left(\sum_{I''\in \mathrm{Minor}(n)\setminus \{I',\emptyset\}} C^{I'}_{I''}[\mathcal{E}^{\theta_2}_{I'',\emptyset}]\right)+[\mathcal{E}^{\theta_2}_{I',J}]+(C^{I'}_{I'}-1)[\mathcal{E}^{\theta_2}_{I',\emptyset}]+C^{I'}_{\emptyset}[1]-C^{I'}_{\emptyset}[\mathcal{E}^{\theta_2}_{\emptyset,\emptyset}],$$  where $C^{I'}_{I''}$ is as before, and for $I'= \emptyset$

$$f'([\mathcal{E}_{\emptyset,J}^{\theta_1}]) = [\mathcal{E}^{\theta_2}_{\emptyset,J}].$$  Using $\Tr_{\theta_{2}} \circ f=\Tr_{\theta_{1}},$ we get the required tracial property of $f'.$ Indeed, $$\Tr^{\Z_2
}_{\theta_{2}} \circ f' ([1])=1=\Tr^{\Z_2
}_{\theta_{1}}([1]),$$ and for $I'= \emptyset$
$$\Tr^{\Z_2
}_{\theta_{2}} \circ f' ([\mathcal{E}_{\emptyset,J}^{\theta_1}])=\frac{1}{2}=\Tr^{\Z_2
}_{\theta_{1}}([\mathcal{E}_{\emptyset,J}^{\theta_1}]),$$  and finally for $I'\neq \emptyset,$

$$\Tr^{\Z_2
}_{\theta_{2}} \circ f' ([\mathcal{E}_{I',J}^{\theta_1}])=\frac{\Tr_{\theta_2}(\sum_{I''\in \mathrm{Minor}(n)} C^{I'}_{I''}[\mathcal{E}^{\theta_2}_{I''}])}{2}=\frac{\Tr_{\theta_{2}} \circ f([\mathcal{E}^{\theta_1}_{I'}])}{2}=\frac{\Tr_{\theta_{1}}([\mathcal{E}^{\theta_1}_{I'}])}{2}=\Tr^{\Z_2
}_{\theta_{1}}([\mathcal{E}_{I',J}^{\theta_1}]).$$

\end{proof}

\appendix

\section{Rieffel-type projections and $\Z_2$-invariance}
\label{sec:riep}
\subsection{The Rieffel projection in $n$-dimensional tori}\label{sec:bottp} 
In this section we will give a description of the map $\psi$, and the construction of the Rieffel projection $e$ used in Theorem~\ref{even n projection}. 

Let us first write down the Morita equivalence construction of Section~\ref{sec:rie} explicitly for $p=1,$ i.e. when $\mathcal{M}$ is the group $\R\times \Z^q,$ $q=n-2.$
As before, write $$\theta=\left(\begin{matrix}
                 \theta_{1,1} &  \theta_{1,2} \\
               \theta_{2,1}& \theta_{2,2}
                 \end{matrix}\right)= \left(\begin{matrix}
               \theta_{1,1} & \theta_{1,2} \\
                -\theta_{1,2}^t &\theta_{2,2}
                 \end{matrix}\right)\in\mathcal{T}_n,$$
                 where
                 \begin{eqnarray}
  \theta_{1,1}=\left(\begin{matrix}
              0 &  \theta_{12} \\
               \theta_{21}& 0
                 \end{matrix}\right)=\left(\begin{matrix}
                 0 &  \theta_{12} \\
              - \theta_{12}& 0
                 \end{matrix}\right)\in\mathcal{T}_2\nonumber
\end{eqnarray} is an invertible $2\times2$ matrix. Then consider the matrix
$T=\left(\begin{matrix}
                  T_{11} & 0 \\
                 0 & {\rm id}_q\\
                 T_{31}&T_{32}
                \end{matrix}
\right),$
where $T_{11}=\left(\begin{matrix}
                  \theta_{12} & 0 \\
                 0 & 1
                 \end{matrix}
\right),$
$T_{31}=\theta_{2,1}$ and $T_{32}=\frac{\theta_{2,2}}{2}.$
Also consider $S=\left(\begin{matrix}
                  S_0 &  -S_0 T_{31}^t \\
                 0& {\rm id}_q\\
                 0& T_{32}^t
                 \end{matrix}
\right),$
where
$S_0=\left(\begin{matrix}
                  0 & 1 \\
                 -\theta_{12}^{-1} & 0
                 \end{matrix}
\right).$
Then the equations \ref{eq:module1},~\ref{eq:module2},~\ref{eq:module3},~\ref{eq:module4}, make $\mathcal{S}(\R \times \Z^q)$  an $A^\infty_{\sigma(\theta)}-A^\infty_{\theta}$
Morita equivalence bimodule, where
$\sigma(\theta)=\left(\begin{matrix}
                  \theta_{1,1}^{-1} & -\theta_{1,1}^{-1}\theta_{1,2} \\
                  \theta_{2,1}\theta_{1,1}^{-1} & \theta_{2,2}-\theta_{2,1}\theta_{1,1}^{-1}\theta_{1,2}
                \end{matrix}
\right).$ Let us denote by $U_1, U_2, \ldots, U_n$ the canonical generators of $A_{\theta}.$ For $A_{\sigma(\theta)},$ we choose the generators $V_1, V_2, \ldots, V_n$ such that  $ \delta_{x_i} \in C^*(\Z^n, \omega_{\sigma(\theta)})$ is identified with $V_i,$ for all $i,$ where $x_i = (0, \ldots, -1, \ldots, 0),$ -1 is at the $i^{th}$ position.  

According to \cite[Proposition 2.1]{Rie81}, if we can find an $f\in \mathcal{S}(\R \times \Z^q)$ such that
$_{A^\infty_{\sigma(\theta)}}\langle f,f\rangle=1,$ then $e=\langle f,f\rangle_{A^\infty_{\theta}}$ is a projection of trace $\theta_{12}$ in $A^\infty_{\theta}\subset A_{\theta},$ and we have an isomorphism $\psi: A_{\sigma(\theta)} \rightarrow eA_{\theta}e,$ given by $\psi(a)= \langle f,af\rangle _{A_{\theta}}.$

Choose a smooth even function $\phi$ on $\R$ as in the definition of $f_{\theta_{12}}$ in \cite[Section 2.2]{Wal21}, assuming $\theta_{12}\in (1/2,1).$ By using a standard regularization argument as in \cite[Lemma 2.1]{Boc97}, we can assume that $\phi$ is smooth and $\sqrt{\phi}$ is also smooth. We restrict $\theta_{12}$ to take values in $(\frac{1}{2},1)$ just to make sure that the projection $e$ that we are going to construct is $\Z_2$-invariant. In general, one can find such $\phi$ with the required properties for a general $\theta_{12}\in (0,1).$ Note that $\phi$ is a compactly supported function (with support in $[-\frac{1}{2},\frac{1}{2}]$) satisfying $\phi(x_0+\theta_{12})=1-\phi(x_0)$ for $-\frac{1}{2} \leq x_0 \leq \frac{1}{2}-\theta_{12}$, $\phi=1$ on $\frac{1}{2}-\theta_{12} \leq x_0 \leq \theta_{12}-\frac{1}{2}$, and $\phi(-x_0)=\phi(x_0)$ for $-\frac{1}{2} \leq x \leq \frac{1}{2}.$

Now define a function $f\in \mathcal{S}(\R \times \Z^q),$ given by $f(x_0,l)=c\sqrt{\phi}(x_0),$ $c=\frac{1}{\sqrt{K\theta_{12}}}$ ($K$ is as in Equation~\ref{eq:module4}), when $\Z^q\ni l=0$ and $f(x,l)=0$ otherwise. Let us first show that $_{A^\infty_{\sigma(\theta)}}\langle f,f\rangle=1.$ To this end, from Equation~\ref{eq:module4}, we have the following:

$$_{A^\infty_{\sigma(\theta)}}\langle f,f\rangle(m)=Ke^{2\pi i\langle S(m),J'S(m)\slash2\rangle}\int_{\R \times \Z^q}\langle x,S''(m)\rangle f(x+S'(m))f(x)dx,$$ for $m=(m_1,m_2,\ldots,m_n) \in \Z^n,$ noting that $f$ is real valued. Using the formula of $S'$ and $f,$ it is clear that the above expression is zero for any non-zero values of $(m_3,m_4,\ldots,m_n)\in \Z^q.$ Also when $m_2\neq 0,$ the above expression for $m=(m_1,m_2,0,\ldots,0)\in \Z^n$ is zero since $x+S'(m)= (x_0+m_2,x_1,\ldots,x_q),$ where $x=(x_0,x_1,\ldots,x_q),$ and $\phi$ is supported in an interval of length one and vanishing at the end points. Finally for an $m= (m_1,0,\ldots,0)$ the expression becomes
\begin{eqnarray*}
&&Kc^2\int_{\R}e^{-\frac{2\pi ix_0 m_1} {\theta_{12}}}\phi(x_0)dx_0\\
&=&c^2K\theta_{12}\int_{\R}e^{-2\pi ix_0 m_1}\phi(x_0\theta_{12})dx_0\\
&=&\int_{\R}e^{-2\pi ix_0 m_1}\widetilde{\phi}(x_0)dx_0,
\end{eqnarray*}
where $\widetilde{\phi}(x_0)=\phi(x_0\theta_{12}).$ Now

$$\int_{\R}e^{-2\pi ix_0 m_1}\widetilde{\phi}(x_0)dx_0=\widehat{\tilde{\phi}}(m_1)=\widehat{\Phi}(m_1),$$
where $~\widehat{}~$ denotes the Fourier transform and $\Phi$ is the periodic function defined by $\Phi(x_0)= \sum\limits_{n\in \Z} \widetilde{\phi}(x_0+n),$ $x_0\in \R.$ But $\sum\limits_{n\in \Z} \widetilde{\phi}(x_0+n)=\sum\limits_{n\in \Z} \phi(\theta_{12}x_0+\theta_{12}n)=1,$ using the defining properties of $\phi,$ for all $x_0\in \R.$ Hence we get $_{A^\infty_{\sigma(\theta)}}\langle f,f\rangle=1.$

Now we want an explicit expression for the projection
 $$\langle f,f\rangle_{A^\infty_{\theta}}(m)=e^{2\pi i\langle -T(m),J'T(m)\slash2\rangle}\int_{\R \times \Z^q}\langle x,-T''(m)\rangle f(x+T'(m))f(x)dx.$$ Let us write $m=(m_1,m_2,\ldots,m_n), x=(x_0,x_1,\ldots,x_q)$ as before. An easy observation shows that the above expression is zero exactly when $m_3=m_4=\cdots =m_n=0.$ This implies that $\langle f,f\rangle_{A^\infty_{\theta}}\in A_{\theta_{12}}\subseteq A_{\theta},$ and a direct computation yields 
 
\begin{equation}\label{eq:rieffel_projection_explicit}
	\langle f,f\rangle_{A^\infty_{\theta}}(m_1,m_2,0,0,\ldots,0)=e^{-\pi i\theta_{12}m_1m_2}\int_{\R}e^{-2\pi ix_om_2} \sqrt{\phi}(x_0+\theta_{12}m_1)\sqrt{\phi}(x_0)dx_0.
\end{equation}

\subsection{Rieffel-type projections in $A_{\theta}\rtimes \Z_2$}\label{sec:bottpincrossed} 
In this section we describe how to construct projections inside $A_{\theta}\rtimes \Z_2$ using the Rieffel projection constructed above, 
and some other technical results used in the proof of Proposition~\ref{prop:imageofi_rieffel}.

Let $f \in \mathcal{S}\left(\mathbb{R} \times \mathbb{Z}^q\right)$ be as above. Since, by construction, $f$ is an even function, it follows from Equation~\ref{eq:proj_mod_inner_holds} that $f$ and the projection $e=\langle f, f\rangle_{A_\theta^{\infty}}$ are $\mathbb{Z}_2$-invariant. Hence the flip-action is an well-defined action on $eA_\theta e.$ Now let us show that the map $\psi: A_{\sigma(\theta)} \rightarrow eA_{\theta}e,$ given by $\psi(a)= \langle f,af\rangle _{A_{\theta}}$ is flip equivariant. To this end, let us use Equation~\ref{eq:proj_mod_inner_holds} and Equation~\ref{eq:proj_mod_T_holds} (for $A_{\sigma(\theta)}$) to check that $\psi(\beta(a))=\beta(\psi(a)).$ This is easy as $\beta(\psi(a))= \beta(\langle f,af\rangle _{A_{\theta}})=\langle fT_g,(af)T_g\rangle _{A_{\theta}}=\langle fT_g,\beta(a)(fT_g)\rangle _{A_{\theta}}=\langle f,\beta(a)f\rangle _{A_{\theta}}=\psi(\beta(a)).$
Hence we get an isomorphism $\psi: A_{\sigma(\theta)}\rtimes \Z_2 \rightarrow eA_{\theta}e\rtimes \Z_2,$ which we denoted by $\psi$ again. Then we have the following commutative diagram:

	\begin{align}\label{diag:i*_commutes}
\xymatrixrowsep{5pc}
\xymatrixcolsep{5pc}
\xymatrix
{
A_{\sigma_\theta} \ar[r]^{i_2} \ar[d]^{\psi} &
A_{\sigma_\theta}\rtimes\Z_2 \ar[d]^{\psi} \\
e A_{\theta}e \ar[r]^{i_2} &
eA_{\theta}e\rtimes \Z_2
}
\end{align}
where $i_2$ is the natural inclusion map.

Next let us understand the image of the projection $\frac{1}{2}(1+V_kW)\in A_{\sigma_\theta}\rtimes\Z_2,$ for $k=3,4,\cdots,n,$ under the map $\psi.$ Of course, we have $\psi(1)=e \in eA_{\theta}e\rtimes \Z_2$. It is then enough to compute $\psi(V_k).$ 

Now from Equation~(\ref{eq:module3}), 
  $$(V_kf)(x)=e^{2\pi i\langle -S(l),J'S(l)\slash2\rangle}\langle x,-S''(l)\rangle f(x+S'(l)),$$ where $l=(0,0, \ldots, -1,\ldots, 0),$ -1 at the $k^{th}$ position. Now $(V_kf)(x)$ is only non-zero when $x_{k-2}=1,$ and $x_i=0,$ for $i\neq 0, k-2.$ Then the value is 
  $$e^{\pi i\frac{\theta_{1k}\theta_{2k}}{\theta_{12}}}e^{2\pi ix_0\frac{-\theta_{1k}}{\theta_{12}}} f(x_0-\theta_{2k},0,0,\ldots,0),$$

  Now we want an explicit expression for 
 $$\langle f,V_kf\rangle_{A^\infty_{\theta}}(m)=e^{2\pi i\langle -T(m),J'T(m)\slash2\rangle}\int_{\R \times \Z^q}\langle x,-T''(m)\rangle (V_kf)(x+T'(m))f(x)dx.$$ 
Now $$(V_kf)(x+T'(m))=(V_kf)(x_0+\theta_{12}m_1,x_1+m_3,x_2+m_4, \ldots, x_q+m_n).$$ Looking at the integral, it is only non-zero when $x_1=x_2=\cdots=x_q=0.$ So for $x=(x_0,0,0,\ldots,0)$
 
 $$(V_kf)(x+T'(m))=(V_kf)(x_0+\theta_{12}m_1,m_3,m_4,\ldots, m_n).$$ But from the previous observation, $V_kf$ is only non-zero when $m_k=1$ and $m_i=0 ~(i \neq 1,k),$ and then the value is 
 $$ e^{\pi i\frac{\theta_{1k}\theta_{2k}}{\theta_{12}}}e^{2\pi i(x_0+\theta_{12}m_1)\frac{-\theta_{1k}}{\theta_{12}}} f(x_0+\theta_{12}m_1-\theta_{2k},0,0).$$
 Finally for $m=(m_1,m_2,0,\ldots, 1,\ldots, 0)$ (1 at the $k^{th}$ position), we have 
 
 
 
 
  $$\langle f,V_kf\rangle_{A^\infty_{\theta}}(m)=e^{2\pi i\langle -T(m),J'T(m)\slash2\rangle}\int_{\R \times \Z^q}\langle x,-T''(m)\rangle (V_kf)(x+T'(m))f(x)dx$$ 
  
  $$= e^{\pi i(-\theta_{12}m_1m_2+\theta_{1k}m_1+\theta_{2k}m_2)}\int_{\R}e^{-2\pi ix_0m_2} e^{\pi i\frac{\theta_{1k}\theta_{2k}}{\theta_{12}}}e^{-2\pi i(x_0+\theta_{12}m_1)\frac{\theta_{1k}}{\theta_{12}}} \sqrt{\phi}(x_0+\theta_{12}m_1-\theta_{2k})\sqrt{\phi}(x_0)dx_0.$$
  
  In particular, the above computation shows that the projection $\psi(\frac{1}{2}(1+V_kW))\in A_{\theta_{12}}\rtimes\Z_2 \subseteq A_{\theta}\rtimes\Z_2,$ where $\Z_2$ in $A_{\theta_{12}}\rtimes\Z_2$ is generated by the self-adjoint unitary $U_kW=:W',$ acting on $A_{\theta_{12}}$ by the flip action (using Lemma~\ref{lemma:flip=flip}). In the above formula of $\langle f,V_kf\rangle_{A^\infty_{\theta}}$ replacing $\theta_{1k}$ and $\theta_{2k}$ by $t\theta_{1k}$ and $t\theta_{2k},$ respectively, for $t\in [0,1],$ we get a homotopy of projections in $A_{\theta_{12}}\rtimes\Z_2$ between $\psi(\frac{1}{2}(1+V_kW))$ (at $t=1$) and $\frac{e}{2}(1+W')$ (at $t=0$). (Note that we have used the fact that for $t=0$, the above expression of $\langle f,V_kf\rangle_{A^\infty_{\theta}}$ matches with the same of $e=\langle f,f\rangle_{A^\infty_{\theta}}$ in Equation~\ref{eq:rieffel_projection_explicit}.)

 \section{strongly totally irrational matrices}\label{sec:strongly_irr}
 
 This section is essentially Appendix I of \cite{CH21}, but with some modifications at the end. For completeness we repeat the whole construction.
 
 Let $s=\{s_i\}_i$ be a sequence of integers such that $s_i>\sum\limits_{j=1}^{i-1}s_j,$ for all $i,$ with $s_1=1.$ We call such a sequence a super-increasing sequence. For $\alpha \in (0,1)$ define the $n \times n$ antisymmetric matrix $\Theta(n)$ by induction: $\Theta(2):=\left(\begin{array}{cc}0 & \alpha^{s_1} \\ -\alpha^{s_1} & 0\end{array}\right);$ $\Theta(n)_{ij} = \Theta(n-1)_{ij},$ for $1<i<j<n,$ and $\Theta(n)_{in}:= \alpha^{s_{p+i}},$ for $i=1,\cdots, n-1,$ where $p=\frac{(n-1)(n-2)}{2}.$ 
Hence
$$\Theta(3)=\left(\begin{array}{ccc}0 & \alpha^{s_1} & \alpha^{s_2}\\ -\alpha^{s_1} & 0 & \alpha^{s_3} \\ -\alpha^{s_2} & -\alpha^{s_3} & 0 \end{array}\right), \quad\Theta(4)=\left(\begin{array}{cccc}0 & \alpha^{s_1} & \alpha^{s_2} & \alpha^{s_4}\\ -\alpha^{s_1} & 0 & \alpha^{s_3} & \alpha^{s_5} \\ -\alpha^{s_2} & -\alpha^{s_3} & 0 & \alpha^{s_6} \\ -\alpha^{s_4} & -\alpha^{s_5}  & -\alpha^{s_6}& 0 \end{array}\right).$$

 \begin{remark}\label{rem:submatrices_super}
 	
 Note that all sub-matrices $\Theta(n)_I$ (as in Definition~\ref{2p minor}) of $\Theta(n)$ is like $\Theta(m),$ for some $m \le n,$ but for a different super increasing sequence $s'=\{s'_i\}_i,$ which is a subsequence of $s.$
   \end{remark} \noindent We then have $\pf(\Theta(2))=\alpha^{s_1},$ and
$$\pf(\Theta(4))=\alpha^{s_1+s_6}-\alpha^{s_2+s_5}+\alpha^{s_4+s_3}.$$
Note that, using the super-increasing property of $s,$ we have that

$$s_1+s_6>s_2+s_5>s_4+s_3.$$

Denote by $s_{i,j}$, the exponent of $\alpha$ in the $ij^{th}$ entry of $\Theta(n).$ Let $s[n]:=s_{(n-1),n}=s_{\frac{n(n-1)}{2}}.$
Let us first recall the following recursive definition of pfaffian from \cite[Page 116]{FP98}. Let $\mathrm{pf}^{i j}(A)$ denote the pfaffian of the $(n-2) \times (n-2)$ skew-symmetric matrix, which we call $A^{ij},$ obtained from $A=(a_{jk})$ by removing the $i^{t h}$, $j^{t h}$ row and the $i^{t h}$, $j^{t h}$ column. Hence $\mathrm{pf}^{i j}(A)=\mathrm{pf}(A^{ij}).$ Let $n$ be even. Then for a fixed integer $j, 1 \leq j \leq n$, one has the following recursive definition of pfaffian

$$
\operatorname{pf}(A)=\sum_{i<j}(-1)^{i+j-1} a_{i j} \operatorname{pf}^{i j}(A)+\sum_{i>j}(-1)^{i+j} a_{i j} \operatorname{pf}^{i j}(A).
$$
In particular, when $j=1,$ we have

\begin{equation}\label{eq:pfaffian_recursive_1}
\operatorname{pf}(A)=\sum_{i>1}(-1)^{i+1} a_{i 1} \operatorname{pf}^{i 1}(A)=\sum_{i>1}(-1)^{i} a_{1i} \operatorname{pf}^{1 i}(A),
\end{equation}
and for $j=n,$ we have

\begin{equation}\label{eq:pfaffian_recursive}
\operatorname{pf}(A)=\sum_{i=1}^{n-1}(-1)^{i+1} a_{i n} \operatorname{pf}^{i n}(A).	
\end{equation}

	\begin{lemma}\label{lem:pfaffian_incresing} For an even $n$ we have
	$$\mathrm{pf}(\Theta(n))=\alpha^{M(n)_1}-\alpha^{M(n)_2}+\alpha^{M(n)_3}-\alpha^{M(n)_4}+\cdots +\alpha^{M(n)_{R(n)}},$$ for a strictly decreasing sequence of numbers  $M(n)_1,M(n)_2,\cdots, M(n)_{R(n)},$ where $R(n)=(n-1)!!$\footnote{double factorial}.
	\end{lemma}

	\begin{proof}
		We prove this by induction on $n.$ As we have seen that the statement is true for $n=2,4$ for all super-increasing sequences. Now assume that the statement is true for $n-2,$ for all super-increasing sequences. Then we must prove that the statement is true for $n,$ for all super-increasing sequences. Fix a super-increasing sequence $s=\{s_i\}_i.$ From Equation~\ref{eq:pfaffian_recursive}, we have
		
		\begin{align*}
			\operatorname{pf}(\Theta(n))&=\sum_{i=1}^{n-1}(-1)^{i+1} \theta_{i n} \operatorname{pf}^{i n}(\Theta)\\
			&= \theta_{(n-1) n} \operatorname{pf}^{(n-1) n}(\Theta)-\theta_{(n-2) n} \operatorname{pf}^{(n-2) n}(\Theta)+\theta_{(n-3) n} \operatorname{pf}^{(n-3) n}(\Theta)-\cdots\\&\quad -\theta_{2 n} \operatorname{pf}^{2 n}(\Theta)+\theta_{1 n} \operatorname{pf}^{1 n}(\Theta)\\
			&= \alpha^{s_{(n-1), n}} \operatorname{pf}^{(n-1) n}(\Theta)-\alpha^{s_{(n-2), n}} \operatorname{pf}^{(n-2) n}(\Theta)+\alpha^{s_{(n-3), n}} \operatorname{pf}^{(n-3) n}(\Theta)-\cdots\\&\textcolor{red}\quad-\alpha^{s_{2, n}} \operatorname{pf}^{2 n}(\Theta)+\alpha^{s_{1, n}} \operatorname{pf}^{1 n}(\Theta).
		\end{align*}
		Now by the induction hypothesis we have that all $\operatorname{pf}^{i n}(\Theta), i=1,2, \ldots ,n-1,$ are of the form as in the statement with $R(n-2)=(n-3)!!.$ Now if we expand all $\operatorname{pf}^{i n}(\Theta), i=1,2, \ldots,n-1,$ keeping the form, we see that the expression $$\alpha^{s_{(n-1), n}} \operatorname{pf}^{(n-1) n}(\Theta)-\alpha^{s_{(n-2), n}} \operatorname{pf}^{(n-2) n}(\Theta)+\alpha^{s_{(n-3), n}} \operatorname{pf}^{(n-3) n}(\Theta)-\cdots\\-\alpha^{s_{2, n}} \operatorname{pf}^{2 n}(\Theta)+\alpha^{s_{1, n}} \operatorname{pf}^{1 n}(\Theta)$$ is already of the required form, using the super-increasing property of $s,$ and noting that the  exponents of $\alpha$ in the expressions of $\operatorname{pf}^{i n}(\Theta), i=1,2, \ldots ,n-1,$ contain no term of the form $s_{*,n}.$
		Note that the total number of terms (after expanding) in the above expression is $(n-1)(n-3)!!=(n-1)!!.$
	\end{proof}

\begin{remark}\label{rem:M_n>}
	It is also clear from above, again using the super-increasing property of $s,$ that $M(n)_{R(n)}>M(n-2)_1.$
\end{remark}

\begin{lemma}\label{lem:pfaffian_in_0_1}

	$$0<\mathrm{pf}(\Theta(n))<\mathrm{pf}(\Theta(n-2))<1,$$ for an even $n.$ \end{lemma}
  \begin{proof}
  Since $$\mathrm{pf}(\Theta(n))=\alpha^{M(n)_1}-\alpha^{M(n)_2}+\alpha^{M(n)_3}-\alpha^{M(n)_4}+\cdots +\alpha^{M(n)_{R(n)}},$$ we have that

  $$\mathrm{pf}(\Theta(n))>\alpha^{M(n)_1}>0,$$ using $-\alpha^{M(n)_{2i}}+\alpha^{M(n)_{2i+1}}>0$ (since $\alpha \in (0,1)$). To show $\mathrm{pf}(\Theta(n))<\mathrm{pf}(\Theta(n-2)),$ we look at the expression $\mathrm{pf}(\Theta(n))-\mathrm{pf}(\Theta(n-2)),$ which is

  $$\alpha^{M(n)_1}-\alpha^{M(n)_2}+\alpha^{M(n)_3}-\cdots +\alpha^{M(n)_{R(n)}}-\alpha^{M(n-2)_1}+\alpha^{M(n-2)_2}-\alpha^{M(n-2)_3}+\cdots -\alpha^{M(n-2)_{R(n-2)}}.$$
  Note that $M(n)_1,M(n)_2,\cdots, M(n)_{R(n)},M(n-2)_1,M(n-2)_2,\cdots, M(n-2)_{R(n-2)}$ is still a strictly decreasing finite sequence due to the above remark. Now using $\alpha^{M(n)_i}-\alpha^{M(n)_{i+1}}<0,$ $\alpha^{M(n)_{R(n)}}-\alpha^{M(n-2)_{1}}<0,$ and $\alpha^{M(n-2)_{2i}}-\alpha^{M(n-2)_{2i+1}}<0,$ we get the above expression less than zero.

  The last inequality in the statement of Lemma~\ref{lem:pfaffian_in_0_1} follows from an argument using induction on $n.$
   \end{proof}

Now let us choose an $\alpha\in (0,1)$ such that $\alpha^{2s[n]}>\frac{1}{2}.$ Then from the above, $M(n)_1<2s[n].$ But $\alpha^{2s[n]}<\alpha^{M(n)_1}.$ So $\frac{1}{2}<\alpha^{M(n)_1}.$ So we have, for such $\alpha$, 
\begin{equation}\label{eq:pf_in_half_1}
	\frac{1}{2}<\mathrm{pf}(\Theta(n))<\mathrm{pf}(\Theta(n-2))<1,\end{equation} for an even $n.$ 

\begin{corollary} With the notations in Section~\ref{sec:rie}, for an $\alpha\in (0,1)$ such that $\alpha^{2s[n]}>\frac{1}{2},$ we have 
\begin{equation}
\frac{1}{2}<\operatorname{pf}\left(F^{j}\left(\Theta(n)_I\right)_{11}\right)<1,
\end{equation}
  for all $I\in \mathrm{Minor}(n)$  with  $2 \leq|I| =2m,$ and for all  $j=0,1, \ldots, m-1.$  	\end{corollary}

\begin{proof}
	For any $I=(i_1,i_2,\ldots,i_{2m})$, since $s^I[n]:=s_{i_{2m-1},i_{2m}}\le s_{(n-1),n}=s[n],$ we still have $\alpha^{2s^I[n]}>\frac{1}{2}.$ Hence from Equation~\ref{eq:pf_in_half_1} (and using Remark~\ref{rem:submatrices_super}), 
	$$\frac{1}{2}<\mathrm{pf}\left(\Theta(n)_I\right)<\mathrm{pf}\left(\Theta(n)_{I''}\right)<1,$$ where $I''$ is obtained from $I$ by deleting the last two numbers. Now the corollary easily follows with the explicit expression (using Lemma \ref{h1}) of $\operatorname{pf}\left(F^{j}\left(\Theta(n)_I\right)_{11}\right)$  in hand.
	\end{proof}

Choose the super-increasing sequence $\{s_i=2^{i-1}\}_i.$ When $\alpha$ is a transcendental number,  it is well known that the numbers
$\alpha,\alpha^2,\dots,\alpha^{2^i},\dots$  as well as any
products of these numbers are linear independent over $\mathbb{Q}$. So we have that $\Theta(n)$ is totally
irrational.
Then by using the above corollary, we get the following corollary.
\begin{corollary}\label{cor:class_of_strongly}
   Let $s$ be the super-increasing sequence $\{s_i=2^{i-1}\}_i$ and let $\alpha\in (0,1)$  be a transcendental number such that $\alpha^{2s[n]}>\frac{1}{2}.$  Let $\Theta(n)$ be the $n\times n$
  antisymmetric matrix involving $\alpha$ and $s.$ Then we have that
  $\Theta(n)$ is a strongly totally irrational matrix for $n\geq2.$
\end{corollary}
\noindent The above corollary gives a large class of examples of strongly totally irrational matrices.

\section{Construction of $C([0,1])\rtimes_{ \Omega_A} \Z^n$}
\label{sec:path}

Let $n$ be an even number. For $I\in \mathrm{Min}(n),$ define $\sigma(I):= i_1+i_2+\dots+i_{2r}$ for $I=(i_1,i_2,\dots,i_{2r})$ and $\sigma(I):= 0$ for $I=\emptyset.$ Also for $I\in \mathrm{Min}(n),$ let $I^{\mathrm{co}}$ be the element in $\mathrm{Min}(n),$ $I^{\mathrm{co}}=(j_1,j_2,\dots,j_{n-2r}),$ such that $\{i_1,i_2,\dots,i_{2r},j_1,j_2,\dots,j_{n-2r}\}=\{1,2,\dots,n\}$ and $\{i_1,i_2,\dots,i_{2r}\}\cap\{j_1,j_2,\dots,j_{n-2r}\}=\emptyset.$ Then we have the \textit{pfaffian summation formula} (\cite[Lemma 4.2]{Ste90})
\begin{equation}\label{eq:sumofpaff}
	\operatorname{pf}(A+B)=\sum_{|I|\le n}(-1)^{\sigma(I)-|I| / 2} \operatorname{pf}(A_{I}) \operatorname{pf}(B_{I^{\mathrm{co}}}),
\end{equation}
 for $A, B \in \mathcal{T}_n.$

Recall that for a matrix $A \in \mathcal{T}_n$ a translate of $A$ is any matrix $A^{\text {tr }}$ such that $A-A^{\text {tr }}$ has only integer entries. In this section we want to prove the following proposition. 

\begin{proposition}\label{prp:translate}
Assume that all the pfaffian minors of $A =(a_{jk})\in \mathcal{T}_n$ are positive, where $n$ is any number. Then there exists a translate of $A,$ $A^{\mathrm{tr}},$ such that all the pfaffian minors of the matrix $(1-t)A^{\mathrm{tr}}+tZ$  are positive for all $t\in [0,1],$ where $Z$ is as in Equation~\ref{eq:def_of_Z}. 
\end{proposition}

Before proving the above proposition, let us try to understand why the above proposition should hold for $n=2, 3, 4.$ For $n=2$ and $3,$ since all the pfaffian minors of $A$ are positive, all pfaffian minors of the matrix $(1-t)A+tZ$  are also clearly positive for all $t\in [0,1]$. Hence we can just take $A^{\mathrm{tr}}=A.$ Now for $n=4,$ if we compute the pfaffian of $(1-t)A+tZ$ using the summation formula above, we get  

\begin{equation}\label{eq:sumofpaff_t_4}
	\operatorname{pf}((1-t)A+tZ)=t^2+(a_{12}+a_{34}-a_{13}-a_{24}+a_{14}+a_{23})t(1-t)+(1-t)^2\pf(A).
\end{equation}
Now we can add a large positive integer $n$ to $a_{12}$ to make $a_{12}+a_{34}-a_{13}-a_{24}+a_{14}+a_{23}$ greater than zero. Now our required $A^{\mathrm{tr}}$ is the matrix obtained by adding $\left(\begin{array}{cccc}0 & n & 0 & 0 \\ -n & 0 & 0 & 0 \\ 0 & 0 & 0 & 0 \\ 0 & 0 & 0 & 0\end{array}\right)$ to $A.$

Let us now consider the matrix $(1-t)A+tZ,$ for $A \in \mathcal{T}_n$ and $n$ is an even number.  Using the pfaffian summation formula we have

\begin{equation}\label{eq:sumofpaff_t}
	\operatorname{pf}((1-t)A+tZ)=\sum_{|I|=2r\le n}(-1)^{\sigma(I)-r} \mathrm{pf}(A_{I}) t^{n/2-r}(1-t)^r.
\end{equation}
 For all $r$ such that $2r\le n,$ set $c_{n,r}:=\sum_{|I|=2r}(-1)^{\sigma(I)-r} \mathrm{pf}(A_{I}).$   Then we claim that the coefficient of $a_{12}$ in $c_{n,r}$ is the coefficient of $t^{n/2-r}(1-t)^r$ in the expression of $\operatorname{pf}^{12}((1-t)A+tZ)$ (refer to the paragraphs which come after Remark~\ref{rem:submatrices_super} for the notation $\operatorname{pf}^{12}$). To show this, first note that 

$$\operatorname{pf}^{12}((1-t)A+tZ)=\sum_{|I'|=2r'}(-1)^{\sigma(I')-r'} \mathrm{pf}(A^{12}_{I'}) t^{(n-2)/2-r'}(1-t)^{r'}.$$ If we put $(n-2)/2-r'=n/2-r,$ we get that $r'=r-1.$ Hence the coefficient of $t^{n/2-r}(1-t)^r$ in the expression of $\operatorname{pf}^{12}((1-t)A+tZ)$ is $\sum_{|I'|=2r-2}(-1)^{\sigma(I')-r+1} \mathrm{pf}(A^{12}_{I'}).$ 
Now in the expression of $c_{n,r},$ only for the $I$'s which are of the form  $(1,2, i_3,i_4,\dots,i_{2r}),$ for some $i_3,i_4,\dots,i_{2r}$ so that $2<i_3<i_4<\dots<i_{2r}\le n,$ $\mathrm{pf}(A_{I})$ will contain a term involving $a_{12}.$ The coefficient of $a_{12}$ in such $\mathrm{pf}(A_{I})$ is $\mathrm{pf}(A_{I''}),$ for $I''=(i_3,i_4,\dots,i_{2r}),$ using the recursive definition (Equation~\ref{eq:pfaffian_recursive_1}) of the pfaffian. But this $I''$ will correspond to $I'=(i_3-2,i_4-2,\dots,i_{2r}-2)$ in $A^{12}.$ Now, it is clear that $(-1)^{\sigma(I)-r} = (-1)^{\sigma(I')-r+1}.$ Since there are $2r-2$ many choices of $I''$ possible, our claim follows. 

For two fixed numbers $l_1, l_2,$ with $l_1\le l_2,$ let us denote by $(l_1,l_2,*,*,\dots,*)$ any such $J\in \mathrm{Min}(n),$ whose first two components are $l_1$ and $l_2,$ i.e. $J=(l_1,l_2, l_3,l_4,\dots,l_{2r}),$ for some $r.$

\begin{proof}[Proof of Proposition~\ref{prp:translate}]
	
Our aim is to show that there exists a translate of $A,$ $A^{\mathrm{tr}},$ such that all pfaffian minors of the matrix  $\mathrm{pf}(((1-t)A^{\mathrm{tr}}+tZ)_J)>0,$ for all $J\in  \mathrm{Minor}(n).$ In fact we will show that there exists a $A^{\mathrm{tr}}$ such that each individual coefficient $c_{|J|,r},~ 2r\le |J|,$ is greater that zero in the expression of $\mathrm{pf}(((1-t)A^{\mathrm{tr}}+tZ)_J)$ for all $J\in  \mathrm{Minor}(n),$ coming from Equation~\ref{eq:sumofpaff_t}.
We do this by induction on the length of $J$. Now clearly the statement is true for all $J$ such that $|J|=2,$ since all the entries of $A$ are positive. Assume that the statement is true for any $J$, such that $|J|\le 2m-2$. In this case we get a translate $A^{\mathrm{tr}},$ which we call by $A$ again, with required properties. We will now show that the statement is true for any $J$, such that $|J|\le 2m.$ Now any such $J$ so that $|J|=2m$ will be of the form $(l_1,l_2,\dots,l_{2m-2},l_{2m}),$ where $1\le l_1 \le l_2\le n-2m+2.$ We perform the following algorithm to get our desired $A^{\mathrm{tr}}.$

 Set $l_1=n-2m+1,$ $l_2= n-2m+2.$ 

\noindent \textbf{Step 1.} Consider any $J= (l_1,l_2,*,*,\dots,*)$ so that $|J|=2m.$ Use the formula Equation~\ref{eq:sumofpaff_t} for $\operatorname{pf}((1-t)A_J+tZ_J)$ and the observation made just before this proof to note that in the coefficients of $t^{m-r}(1-t)^r, r=1,2,\dots,m,$ the coefficients of $a_{l_1l_2}$ are positive using the induction hypothesis. By adding a number to $a_{l_1l_2}$ we can make sure that the coefficients of $t^{m-r}(1-t)^r, r=1,2,\dots,m,$ are positive. We can do this for each  $J= (l_1,l_2,*,*,\dots,*),$ and hence can add a sufficiently large number to $a_{l_1l_2}$ such that the coefficients of $t^{m-r}(1-t)^r, r=1,2,\dots,m,$ are positive in all $\operatorname{pf}((1-t)A_J+tZ_J),$ for all $J= (l_1,l_2,*,*,\dots,*).$ Thus by altering $a_{l_1l_2}$ as above, we get a translate of $A$, which we call by $A$ again. If $l_1>1,$ then set $l_1:= l_1-1, l_2:=l_2,$ and if $l_1=1,$ then set $l_1:= l_2-2, l_2:=l_2-1.$ 
Now if $l_1=1,$ $l_2= 2,$ go to Step 3. Otherwise, go to Step 2.

\noindent \textbf{Step 2.}  As in Step 1 add a sufficiently large number to $a_{l_1l_2},$ such that the coefficients
$t^{m-1-r}(1-t)^r,~ r=1,2,\dots,m-1,$ are positive in all $\operatorname{pf}((1-t)A_J+tZ_J),$ for all $J= (l_1,l_2,*,*,\dots,*)$ such that 
 $|J|=2m-2.$ Then repeat this  for $J= (l_1,l_2,*,*,\dots,*),$ such that $|J|= 2m-4$ and continue the process until we reach such $J$'s so that $|J|=4.$ After the above alterations, we get a translate of $A$, which we call by $A$ again. Go to Step 1.

\noindent \textbf{Step 3.}  Again, as before, add a sufficiently large number to $a_{12},$ such that
$t^{m-r}(1-t)^r,~ r=1,2,\dots,m,$ are positive in all $\operatorname{pf}((1-t)A_J+tZ_J),$ for all $J= (1,2,*,*,\dots,*)$ such that 
 $|J|=2m.$    
 
 After the above step we end up with a translate $A^{\mathrm{tr}}$ of $A.$ Now we claim that 
$\operatorname{pf}((1-t)A^{\mathrm{tr}}_J+tZ_J)$ is positive for all $|J|\le2m.$ Let $J=(l_1,l_2,*,*,\dots,*)$ be fixed such that $|J|\le2m.$ If $l_2>n-2m+2,$ $|J|\le2m-2.$ In this case, no entries of $A_J$ have been altered in the above algorithm, and hence  $\operatorname{pf}((1-t)A^{\mathrm{tr}}_J+tZ_J)=\operatorname{pf}((1-t)A_J+tZ_J)$ is positive by the induction hypothesis. If $l_2\le n-2m+2,$ after one alteration of $a_{l_1l_2}$ (let us call the translate $A'$ so that we have $\operatorname{pf}((1-t)A'_J+tZ_J)>0,)$ no alterations have been occurred for the entries of $A'_J$ except possible alterations of  $a_{l_1l_2}$ which still keep $\operatorname{pf}((1-t)A'_J+tZ_J)>0.$ Hence $\operatorname{pf}((1-t)A^{\mathrm{tr}}_J+tZ_J)>0.$             
       
\end{proof} 
 
\section*{Acknowledgements}

This research was  supported by DST, Government of India under the \emph{DST-INSPIRE
Faculty Scheme} with Faculty Reg. No. IFA19-MA139.


\end{document}